\documentclass{article}

\usepackage{amsthm}
\usepackage{amsmath, amssymb}
\usepackage{hyperref}
\usepackage{tikz}
\usetikzlibrary{matrix}
\usepackage[matrix, arrow, curve]{xy}
\usepackage{xcolor}

\newtheorem{Thm}{Theorem}[section]
\newtheorem{Def}[Thm]{Definition}
\newtheorem{Lem}[Thm]{Lemma}
\newtheorem{Prop}[Thm]{Proposition}
\newtheorem{Kor}[Thm]{Corollary}
\newtheorem{Rem}[Thm]{Remark}
\newtheorem{Bsp}[Thm]{Example}

\newtheorem{Assumption}[Thm]{General Assumptions}
\newtheorem{Conjecture}[Thm]{Conjecture}

\title{Motivic cohomology of cyclic coverings}

\author{Tariq Syed\\
Mathematisches Institut\\
Heinrich-Heine-Universit{\"a}t D{\"u}sseldorf\\
Universit{\"a}tsstra{\ss}e 1\\
40225 D{\"u}sseldorf, Germany\\
tariq.syed@gmx.de}

\date{\today}
\begin{document}

\maketitle

\begin{abstract}
Cyclic coverings produce many examples of topologically contractible smooth affine complex varieties. In this paper, we study the motivic cohomology groups of cyclic coverings over algebraically closed fields of characteristic $0$. In particular, we prove that in many situations Chow groups of cyclic coverings become trivial after tensoring with $\mathbb{Q}$. Furthermore, we can prove that the Chow groups of certain bicyclic coverings are trivial even without tensoring with $\mathbb{Q}$.
\end{abstract}

\tableofcontents

\section{Introduction}

The generalized Serre question on algebraic vector bundles asks whether algebraic vector bundles over topologically contractible smooth affine $\mathbb{C}$-varieties are trivial or not (cf. \cite[Question 6]{AO}). While it was known before that the generalized Serre question has a positive answer in dimensions $\leq 2$, it remains open in higher dimensions (cf. \cite[Section 5.5.2]{AO}). General classification results in dimension $3$ imply that a topologically contractible smooth affine complex threefold $X$ has only trivial vector bundles if and only if the Chow groups $CH^2 (X)$ and $CH^3 (X)$ are trivial (cf. \cite{AF1}, \cite{KM}); the main result in \cite{Sy} implies that all vector bundles over a topologically contractible smooth affine complex fourfold $X$ are trivial if $CH^2 (X)$, $CH^3 (X)$, $CH^4 (X)$ and the Nisnevich cohomology group $H_{Nis}^{2}(X, \textbf{I}^{3})$ are trivial. Note that Chow groups are examples of motivic cohomology groups (cf. \cite[Corollary 19.2]{MVW}) and that there is a link between the group $H_{Nis}^{2}(X, \textbf{I}^{3})$ and motivic cohomology groups with $\mathbb{Z}/2\mathbb{Z}$-coefficients given by the long exact sequence in \cite[Theorem 1.3]{Tot}. A conjecture of Asok-{\O}stv{\ae}r suggests that topologically contractible smooth affine complex varieties should be stably $\mathbb{A}^1$-contractible (cf. \cite[Conjecture 5.5.3.11]{AO}). In view of the cohomological criteria mentioned above, this would automatically yield a positive answer to the generalized Serre question in dimensions $3$ and $4$. It is well-known that stably $\mathbb{A}^1$-contractible smooth affine complex varieties have the integral motivic cohomology of the base field; on the other hand, there exists a criterion for the stable $\mathbb{A}^1$-contractibility of a smooth affine complex variety in terms of motivic cohomology groups: Indeed, if $X$ is a smooth affine complex variety and $x \in X$ is a closed point of $X$ such that for every smooth affine $\mathbb{C}$-scheme $Y$ the projection maps $pr_Y : X \times_{\mathbb{C}} Y \rightarrow Y$ induce isomorphisms on motivic cohomology groups, then $(X,x)$ is stably $\mathbb{A}^1$-contractible (cf. \cite[Theorem 3.1]{DPO}). Altogether, it is of interest to study motivic invariants and particularly the motivic cohomology of topologically contractible smooth affine complex varieties.\\
The literature (cf. \cite{Z}, \cite{AO}) provides many examples of topologically contractible smooth affine complex varieties; most of these examples are given by affine modifications (cf. \cite{KZ},\cite[Section 4]{Z}) or by cyclic coverings (cf. \cite[Section 5]{Z}). The motivic cohomology and motivic topology of affine modifications have been studied in \cite{DPO} and in \cite{ADO}, while no general results are known on the motivic cohomology of cyclic coverings. In this paper, we prove the first general results on the motivic cohomology of cyclic coverings; the underlying notion of cyclic coverings in this paper is directly taken from \cite[Section 5]{Z}:

\begin{Def}\label{CyclicCoverings}
Let $k$ be a field of characteristic $0$, let $X$ be an affine $k$-variety and let $f \in \mathcal{O}_{X}(X) \setminus \{0\}$ be a regular function and let $s > 1$ be an integer. Then
\begin{itemize}
\item we denote by $F_0$ the closed subscheme of $X$ defined by $f$;
\item we let $Y_{s}$ be the closed subscheme of $X \times_k \mathbb{A}^1_k$ defined by $f = u^s$, where $u$ is the variable of $\mathbb{A}^1_k$;
\item we will denote by $F$ the closed subscheme of $Y_s$ defined by $u=0$.
\end{itemize}
We call the projection $\varphi_s : Y_s \rightarrow X$ a cyclic covering of $X$ branched to order $s$ along $F_0$.
\end{Def}

\begin{Assumption}\label{Assumptions}
We furthermore assume that
\begin{itemize}
\item $k$ is algebraically closed;
\item $u^s - f$ is prime in the polynomial rings $\mathcal{O}_{X}(X)[u]$ and $k(X)[u]$, where $k(X)$ is the field of fractions of the domain $\mathcal{O}_{X}(X)$;
\item $X$ and $F_0$ (and hence $Y_s$ by the remarks in \cite[Definition 5.1]{Z}) are smooth over $k$.
\end{itemize}
\end{Assumption}

We primarily study the motivic cohomology groups of the variety $Y_{s}$ with coefficients in a commutative ring $R$ such that $s \in R^{\times}$ and the integral motivic cohomology groups tensored with $R$. For this purpose, we first focus on the induced morphism $Y_{s} \setminus F \rightarrow X \setminus F_{0}$; by abuse of notation, we also denote this morphism by $\varphi_s$. This induced morphism is a finite {\'e}tale Galois cover whose Galois group is the group $\mu_{s}(k)$ of $s$th roots of unity. Motivated by this, we then study the maps on motivic cohomology groups induced by finite {\'e}tale Galois covers. For a finite {\'e}tale Galois cover $\varphi: U \rightarrow V$ between smooth $k$-varieties with Galois group $G$ and of degree $s = |G|$, we prove that $\varphi$ induces an isomorphism ${H^{p,q}(U,R)}^{G} \cong H^{p,q}(V,R)$ whenever $s \in R^{\times}$ (cf. Corollary \ref{FixedSubgroup}). Here ${H^{p,q}(U,R)}^{G}$ denotes the subgroup of elements in $H^{p,q}(U,R)$ fixed by the Galois action of $G$. This result applies in particular to the morphism $\varphi_{s}: Y_{s} \setminus F \rightarrow X \setminus F_{0}$. Motivated by this, we further investigate under which circumstances the Galois action on $H^{p,q}(Y_s \setminus F,R)$ is in fact trivial. While we are able to prove a general result on motivic cohomology groups (cf. Theorem \ref{Ind-Thm}) and, in particular, some results on the motivic cohomology groups of $Y_s$ with $\mathbb{Z}/n\mathbb{Z}$-coefficients for $n \geq 1$ (cf. Corollary \ref{FiniteCoefficients}), the main results proven in this paper concern the Chow groups of a cyclic covering (cf. Theorem \ref{Chow}):

\begin{Thm}
With the notation of Definition \ref{CyclicCoverings} and under the General Assumptions \ref{Assumptions}, further assume that
\begin{itemize}
\item[(a)] $R$ is a coefficient ring with $s \in R^{\times}$;
\item[(b)] there is a $\mathbb{G}_{m,k}$-action $\gamma$ on $X$ which makes $f$ a quasi-invariant of weight $d \in \mathbb{Z}$ with respect to $\gamma$ and $\langle d,s \rangle = \mathbb{Z}$.
%\item[(c)] $CH^i (X)\otimes_{\mathbb{Z}}R = 0$ for $i \geq 1$
\end{itemize}
Then $\varphi_{s}: Y_{s} \rightarrow X$ induces isomorphisms $H^{2i,i} (X,R) \cong H^{2i,i} (Y_{s},R)$ for $i \in \mathbb{Z}$.
\end{Thm}

In particular, as $H^{2i,i}(Y,\mathbb{Q}) \cong CH^{i}(Y) \otimes_{\mathbb{Z}}\mathbb{Q}$ (cf. \cite[Preface]{MVW}) for $Y \in Sm_k$, the Chow groups of $Y_{s}$ are torsion if those of $X$ are torsion. The second assumption in the theorem also appears similarly in classical results on the singular homology of cyclic coverings (cf. \cite[Theorem 5.2]{Z}): As $X$ is affine, a $\mathbb{G}_{m,k}$-action on $X$ is uniquely determined by its associated homomorphism of $k$-algebras $\mathcal{O}_{X}(X) \rightarrow \mathcal{O}_{X}(X)[\lambda, \lambda^{-1}]$, where $\mathcal{O}_{X}(X)[\lambda, \lambda^{-1}]$ denotes the ring of Laurent polynomials over $\mathcal{O}_{X}(X)$; then a regular function $f \in \mathcal{O}_{X}(X)$ is a quasi-invariant of weight $d \in \mathbb{Z}$ with respect to the $\mathbb{G}_{m,k}$-action if and only if this homomorphism of $k$-algebras sends $f$ to $f \lambda^{d}$. For example, when $X = \mathbb{A}^n_{k} = Spec (k[x_{1},...,x_{n}])$ and the $\mathbb{G}_{m,k}$-action is given by $(\lambda, x_{1},...,x_{n}) \mapsto (x_{1} \lambda^{d_{1}},...,x_{n} \lambda^{d_{n}})$ for some $d_{j} \in \mathbb{Z}$, then any coordinate function $x_{j}$ is a quasi-invariant of weight $d_{j}$. While one needs a coefficient ring $R$ with $s \in R^{\times}$ in the theorem above, one can even prove results on the integral motivic cohomology groups $H^{2i,i}(Y_{s},\mathbb{Z})$ (i.e., on the Chow groups) of bicyclic coverings as considered in \cite[Example 6.2]{Z} (cf. Theorem \ref{Bicyclic}):

\begin{Thm}\label{Bicyclic-Intro}
With the notation of Definition \ref{CyclicCoverings} and under the General Assumptions \ref{Assumptions}, further assume that
\begin{itemize}
\item[(a)] $f \in O_X (X)$ is prime and $O_X (X)$ is a UFD;
\item[(b)] the smooth variety $X$ has the Chow groups of $Spec(k)$;
\item[(c)] the smooth variety $F_{0}=\{f=0\} \subset X$ has the Chow groups of $Spec(k)$;
\item[(d)] there is a $\mathbb{G}_{m,k}$-action $\gamma$ on $X$, which makes $f$ a quasi-invariant of weight $d \in \mathbb{Z}$ with respect to $\gamma$;
\item[(e)] $\langle d,s \rangle = \langle d,t \rangle = \langle s,t \rangle= \mathbb{Z}$.
%\item[(f)] $CH^1 (Y_{s}) \cong CH^1 (Y_{t}) = 0$
\end{itemize}
Then $CH^i (Y_{s,t})=0$ for $i \geq 1$, where $Y_{s,t}:=\{f + u^{s} + v^{t}=0\} \subset X \times_k \mathbb{A}^1_k \times_k \mathbb{A}^1_k$.
\end{Thm}

If $k= \mathbb{C}$ and $X$ is topologically contractible in the theorem, then the variety $Y_{s,t}$ is actually known to be topologically contractible by \cite[Example 6.2]{Z}. Note that the theorem implies that the varieties $Y_{s,t}$ have only stably trivial algebraic vector bundles. If $\dim (Y_{s,t})=3$, then general classification results immediately imply that all algebraic vector bundles over $Y_{s,t}$ are trivial (cf. \cite{AF1}, \cite{KM}); if $\dim (Y_{s,t})=4$, then all algebraic vector bundles over $Y_{s,t}$ are trivial as soon as the Nisnevich cohomology group $H_{Nis}^{2}(Y_{s,t},\textbf{I}^3)$ is trivial (cf. \cite{Sy}). We are able to prove that $H_{Nis}^{2}(Y_{s,t},\textbf{I}^3)$ vanishes under the assumptions of Theorem \ref{Bicyclic-Intro} if $\dim (Y_{s,t})=4$ and hence all vector bundles over $Y_{s,t}$ are indeed trivial (cf. Theorem \ref{Bicyclic-VB}). Altogether, our results seem to indicate a positive answer to the generalized Serre question in dimensions 3 and 4.\\
Finally, we emphasize that our results make the motivic cohomology groups and particularly the Chow groups of \textbf{many} concrete examples of topologically contractible smooth affine complex varieties very computable. As an illustration of this, we present several exemplary computations of interest in this paper (cf. Section 4); the examples include Koras-Russell threefolds of the third kind which are obtained as cyclic coverings of Koras-Russell threefolds of the first and second kinds. Koras-Russell threefolds (cf. \cite{KR}) of the first and second kinds were studied by many mathematicians (e.g., in \cite{HKO} and \cite{DF}) and have significantly shaped the research on topologically contractible smooth affine $\mathbb{C}$-varieties in the last decades.\\
The organization of the paper is as follows: Sections \ref{2.1}, \ref{2.2} and \ref{2.3} serve as a brief introduction to motivic cohomology and motivic homotopy theory as needed for this paper; Section \ref{2.4} gives a brief overview of classification results on algebraic vector bundles over topologically contractible smooth affine complex varieties. Altogether, Section \ref{2} serves as a motivation to study the motivic cohomology of cyclic coverings. In Section \ref{3} we introduce cyclic coverings, study their motivic cohomology and prove the main results of this paper. Finally, in Section \ref{Examples} we examine several examples of smooth affine varieties which are obtained as cyclic coverings of well-known smooth affine varieties.

\subsection*{Acknowledgements}
The author would like to thank the anonymous referee for suggesting changes which greatly improved the exposition of the paper. The author would like to thank Alexey Ananyevskiy, Aravind Asok, Fr{\'e}d{\'e}ric D{\'e}glise, Adrien Dubouloz, Jean Fasel, Lucy Moser-Jauslin, Joaqu{\'i}n Moraga, Morgan Opie, Sabrina Pauli, Charlie Petitjean, Brian Shin, Charles Vial, Nanjun Yang and Paul Arne {\O}stv{\ae}r for helpful comments on this work. The author was funded by the Deutsche Forschungsgemeinschaft (DFG, German Research Foundation) - Project number 461453992.

\section{Background}\label{Background}\label{2}

\subsection{Motivic cohomology}\label{Motivic cohomology}\label{2.1}

In this section, we briefly discuss some basic facts about motivic cohomology as needed for this paper. Our main reference is \cite{MVW}.\\
We fix a perfect base field $k$ and we let $Sm_k$ be the category of smooth separated schemes of finite type over $k$. Recall that there is an additive category $Cor_k$ whose objects are smooth separated schemes of finite type over $k$ and whose morphisms are called finite correspondences (cf. \cite[Lecture 1]{MVW}). An elementary correspondence between $X,Y \in Sm_k$ with $X$ connected is an irreducible closed subset of $X \times_k Y$ whose associated integral closed subscheme is finite and surjective over $X$; if $X$ is not connected, then an elementary correspondence $X$ to $Y$ is an elementary correspondence from a connected component of $X$ to $Y$. Then one defines the group of finite correspondences $Cor_k (X,Y)$ as the free abelian group generated by elementary correspondences from $X$ to $Y$, so a finite correspondence from $X$ to $Y$ is a formal sum of elementary correspondences from $X$ to $Y$. The composition law in $Cor_k$ is defined by using the intersection product of cycles (cf. \cite[Remarks before Definition 1.5]{MVW}: Given $X,Y,Z \in Sm_k$ and elementary correspondences $V \in Cor_k (X,Y)$ and $W \in Cor_k (Y,Z)$, one defines $W \circ V$ as the pushforward of the intersection product $(V \times_k Z) \cdot (X \times_k W)$ along the projection $p : X \times_k Y \times_k Z \rightarrow X \times_k Z$; the intersection product is defined in \cite[Appendix 17.A]{MVW}, the pushforward just before \cite[Lemma 1.4]{MVW}. There is a faithful embedding

\begin{center}
$Sm_k \rightarrow Cor_k$.
\end{center}

This functor sends $X \in Sm_k$ to its corresponding object in $Cor_k$ and a morphism $f: X \rightarrow Y$ between smooth $k$-schemes $X,Y \in Sm_k$ is sent to its graph $\Gamma_f$ (which is a finite correspondence from $X$ to $Y$ by \cite[Example 1.2]{MVW}). The coproduct in $Cor_k$ is just the disjoint union of schemes. Then one defines presheaves with transfers as follows:

\begin{Def}
A presheaf with transfers is an additive functor $Cor_k^{op} \rightarrow Ab$. Analogously, for any commutative ring $R$, a presheaf with transfers of $R$-modules is an additive functor $Cor_k^{op} \rightarrow \textit{R-mod}$.
\end{Def}

\begin{Bsp}
For any $X \in Sm_k$, the representable presheaf with transfers given by $U \mapsto Cor_k (U,X)$ is denoted by $\mathbb{Z}_{tr}(X)$. Similarly, the representable presheaf with transfers of $R$-modules given by $U \mapsto Cor_k (U,X) \otimes_{\mathbb{Z}} R$ is denoted by $R_{tr}(X)$.
\end{Bsp}

Note that any presheaf with transfers (of $R$-modules) can be restricted to a presheaf $Sm_k^{op} \rightarrow Ab$ (to a presheaf $Sm_k^{op} \rightarrow \textit{R-mod}$). In particular, one has a natural definition for sheaves with transfers:

\begin{Def}
A presheaf with transfers is a Nisnevich (resp. {\'e}tale) sheaf with transfers if its restriction to $Sm_k$ is a Nisnevich (resp. {\'e}tale) sheaf.
\end{Def}

Representable presheaves with transfers are known to be {\'e}tale sheaves with transfers:

\begin{Lem}[{{\cite[Lemma 6.2]{MVW}}}]
For any $X \in Sm_k$, the presheaf with transfers $\mathbb{Z}_{tr}(X)$ is an {\'e}tale sheaf with transfers.
\end{Lem}

Recall that one defines complexes of presheaves with transfers $\mathbb{Z}(q)$, $q \in \mathbb{Z}$, called motivic complexes (cf. \cite[Definition 3.1]{MVW}). These complexes are used in order to define motivic cohomology groups (cf. \cite[Definition 3.4]{MVW}):

\begin{Def}
The motivic cohomology of $X \in Sm_k$ is defined as the Nisnevich hypercohomology $H^{p,q}(X,\mathbb{Z}):=\mathbb{H}^p_{Nis} (X, \mathbb{Z}(q))$. If $R$ is a commutative ring, one sets $R(q) = \mathbb{Z}(q) \otimes_{\mathbb{Z}} R$. The motivic cohomology of $X \in Sm_k$ with coefficients in $R$ is defined as the Nisnevich hypercohomology $H^{p,q}(X,R):=\mathbb{H}^p_{Nis} (X, R(q))$.
\end{Def}

Motivic cohomology groups are contravariantly functorial with respect to morphisms in $Cor_k$ and hence define presheaves with transfers (cf. \cite[Example 13.11]{MVW}). The triangulated category $\textbf{DM}^{eff,-}_{Nis}(k,R)$ of effective motives over $k$ with $R$-coefficients is defined as the localization of the derived category $D^{-}Sh_{Nis}(Cor_{k},R)$ of Nisnevich sheaves with transfers of $R$-modules with respect to $\mathbb{A}^1$-weak equivalences (cf. \cite[Definition 14.1]{MVW}). Motivic cohomology with coefficients in $R$ is representable in $\textbf{DM}^{eff,-}_{Nis}(k,R)$, i.e., there are natural isomorphisms

\begin{center}
$H^{p,q}(X, R) \cong Hom_{\textbf{DM}^{eff,-}_{Nis}(k,R)}(M(X),R(q)[p])$
\end{center}

for any $X \in Sm_k$ (cf. \cite[Proposition 14.16]{MVW}). Here $M(X)$ is the motive of $X$, i.e., the class of $R_{tr}(X)$ in $\textbf{DM}^{eff,-}_{Nix}(k,R)$.\\
If $X \in Sm_k$ has a smooth closed subscheme $Z$ of codimension $c$ with complement $U = X \setminus Z$, then there is a distinguished triangle called Gysin triangle (cf. \cite[Properties 14.5]{MVW}) of the form

\begin{center}
$M(U) \rightarrow M(X) \rightarrow M(Z)(c)[2c] \rightarrow M(U)[1]$.
\end{center}

This distinguished triangle induces a long exact localization sequence of motivic cohomology groups of the form

\begin{center}
$...\rightarrow H^{p-2c,q-c}(F,R) \rightarrow H^{p,q}(X,R) \rightarrow H^{p,q}(U, R) \rightarrow H^{p+1-2c,q-c}(F,R) \rightarrow ...$
\end{center}
The Gysin triangle and its associated long exact localization sequence are functorial with respect to any commutative diagram

\begin{center}
$\begin{xy}
  \xymatrix{
     T \ar[d]_g \ar[r] & Y \ar[d]^f \\
     Z \ar[r] & X}
\end{xy}$
\end{center}

such that $Z$ is a smooth closed subscheme of $X \in Sm_k$, $T$ is a smooth closed subscheme of $Y \in Sm_k$ and the underlying diagram of topological spaces is cartesian (e.g., see \cite{D}). Finally, recall that there are natural isomorphisms

\begin{center}
$H^{2p,p}(X, \mathbb{Z}) \cong CH^p (X)$
\end{center}

for any $X \in Sm_k$ and integer $p \geq 0$ (cf. \cite[Corollary 19.2]{MVW}). With this identification, the Gysin triangle above recovers the usual localization sequences of Chow groups.

\subsection{$\mathbb{A}^1$-contractible varieties}\label{2.2}

We outline the construction of the unstable $\mathbb{A}^1$-homotopy category over a field given in \cite{MV}. The underlying idea of $\mathbb{A}^{1}$-homotopy theory is to develop a homotopy theory of schemes in which the affine line $\mathbb{A}^1$ plays the role of the unit interval $[0,1]$ in topology.\\
So let $k$ be a field. Following \cite{MV}, we let $Sm_k$ be the category of smooth separated schemes of finite type over $k$ and we consider the category $Spc_{k}=\Delta^{op}Shv_{Nis}(Sm_{k})$ of simplicial Nisnevich sheaves over $Sm_k$; simplicial Nisnevich sheaves over $Sm_k$ are also referred to as spaces. Note that both the category $Sm_k$ and the category $\Delta^{op} Sets$ of simplicial sets can be embedded into $Spc_{k}$. There is a model structure on this category in which cofibrations are simply given by monomorphisms and weak equivalences are given by morphisms which induce weak equivalences of simplicials sets on stalks; this model structure is called the simplicial model structure and weak equivalences with respect to this model structure are also called simplicial weak equivalences. The $\mathbb{A}^1$-model structure is obtained as a left Bousfield localization of the simplicial model structure with respect to the projection morphisms $\mathcal{X} \times_k \mathbb{A}^{1} \rightarrow \mathcal{X}$; its weak equivalences are also called $\mathbb{A}^{1}$-weak equivalences. Its associated homotopy category (obtained from $Spc_{k}$ by inverting $\mathbb{A}^{1}$-weak equivalences) is usually denoted $\mathcal{H} (k)$ and called the unstable $\mathbb{A}^{1}$-homotopy category over $k$. This category also has a pointed version which is constructed completely analogously by considering the category $Spc_{k, \bullet}$ of pointed simplicial Nisnevich sheaves over $Sm_k$ (which are referred to as pointed spaces); it is denoted $\mathcal{H}_{\bullet} (k)$ and called the pointed unstable $\mathbb{A}^1$-homotopy category.\\% The category $Spc_{k, \bullet}$ is a pointed model category and hence features the formalism of fiber and cofiber sequences (cf. \cite[Chapter 6]{Ho}).\\
For spaces $\mathcal{X},\mathcal{Y}$, we denote by $[\mathcal{X},\mathcal{Y}]_{\mathcal{H}(k)}$ the set of morphisms from $\mathcal{X}$ to $\mathcal{Y}$ in $\mathcal{H}(k)$; similarly, for pointed spaces $(\mathcal{X},x),(\mathcal{Y},y)$, we denote by $[(\mathcal{X},x),(\mathcal{Y},y)]_{\mathcal{H}_{\bullet}(k)}$ the set of morphisms from $(\mathcal{X},x)$ to $(\mathcal{Y},y)$ in $\mathcal{H}_{\bullet}(k)$.\\
By analogy with topology, one can define a smash product $(\mathcal{X},x) \wedge (\mathcal{Y},y)$ for pointed spaces $(\mathcal{X},x),(\mathcal{Y},y)$. The functor $\Sigma_{s}=(S^1,\ast) \wedge -: Spc_{k, \bullet} \rightarrow Spc_{k, \bullet}$ is called the simplicial suspension functor; it has a right adjoint functor $\Omega_{s}: Spc_{k, \bullet} \rightarrow Spc_{k, \bullet}$ called the loop space functor. The pair of functors forms a Quillen adjunction.

\begin{Bsp}[{{\cite[Lemma 5.3.1.3]{AO}}}]\label{VB-weak-equivalence}
Any Nisnevich locally trivial $\mathbb{A}^n_k$-bundle $\pi: E \rightarrow X$ is an $\mathbb{A}^{1}$-weak equivalence. In particular, this holds for any geometric vector bundle $\pi: E \rightarrow X$.
\end{Bsp}

With a homotopy theory for schemes at hand, it is natural to define the following notion of contractibility.

\begin{Def}
A space $\mathcal{X} \in Spc_k$ is called $\mathbb{A}^{1}$-contractible if the unique morphism $\mathcal{X} \rightarrow Spec(k)$ is an $\mathbb{A}^{1}$-weak equivalence.
\end{Def}

\begin{Bsp}
The affine spaces $\mathbb{A}^n_k$ are $\mathbb{A}^{1}$-contractible. This follows from Example \ref{VB-weak-equivalence}.
\end{Bsp}

If $k$ is a field of characteristic $0$, then in low dimensions affine spaces are the unique $\mathbb{A}^{1}$-contractible smooth schemes:

\begin{Thm}[{{\cite[Claim 5.7]{AD}}}]
If $char(k) =0$, the affine line $\mathbb{A}^1_k$ is the unique $\mathbb{A}^{1}$-contractible smooth scheme of dimension $1$.
\end{Thm}

\begin{Thm}[{{\cite[Theorem 1.1]{CR}}}]
If $char (k) = 0$, the affine plane $\mathbb{A}^2_{k}$ is the unique $\mathbb{A}^{1}$-contractible smooth scheme of dimension $2$.
\end{Thm}

\subsection{Stably $\mathbb{A}^1$-contractible and topologically contractible varieties}\label{2.3}

In classical homotopy theory computations become more accessible after inverting the suspension functor for pointed spaces and $\mathbb{A}^{1}$-homotopy theory is no different in this respect. By analogy, there is a stable motivic homotopy category $\mathcal{SH}(k)$ of $\mathbb{P}^1$-spectra, where $k$ is a field as usual; instead of explaining the construction of $\mathcal{SH}(k)$ and presenting its basic properties, we just give a quick definition of stable $\mathbb{A}^{1}$-contractibility which is completely sufficient for this paper.

\begin{Def}
Let $X \in Sm_{k}$ and $x$ be a closed point. Then the pointed space $(X,x)$ is stably $\mathbb{A}^{1}$-contractible if ${\mathbb{P}_{k}^{1}}^{\wedge n} \wedge (X,x)$ is an $\mathbb{A}^{1}$-contractible space for some $n \geq 0$.
\end{Def}

Naturally, $\mathbb{A}^1$-contractible spaces are automatically stably $\mathbb{A}^1$-contractible by definition (for any choice of a basepoint). Now let us work over a base field $k$ which admits an embedding $\iota: k \hookrightarrow \mathbb{C}$ into the field of complex numbers. By means of such an embedding, one may associate a complex manifold $X_{\iota}^{an}$ to any smooth variety $X$ over $k$.

\begin{Def}
Let $k$ be a field which admits an embedding into $\mathbb{C}$. A smooth affine $k$-variety $X \in Sm_k$ is called topologically contractible if the manifold $X_{\iota}^{an}$ is a contractible topological space for any embedding $\iota: k \hookrightarrow \mathbb{C}$.
\end{Def}

One should think of topological contractibility as a significantly weaker notion than $\mathbb{A}^{1}$-contractibility. If the base field $k$ admits an embedding into $\mathbb{C}$, say $\iota: k \hookrightarrow \mathbb{C}$, then the assignment, $X \mapsto X_{\iota}^{an}$ can be extended to a complex or topological realization functor
\begin{center}
$\mathfrak{R}_{\iota}:\mathcal{H}(k) \rightarrow \mathcal{H}$,
\end{center}
where $\mathcal{H}$ denotes the homotopy category of topological spaces (cf. \cite[Section 3.3]{MV}). In particular, $\mathbb{A}^{1}$-contractible smooth schemes are also topologically contractible. Conversely, the work of Choudhury-Roy (cf. \cite{CR}) shows that not every topologically contractible smooth complex variety is $\mathbb{A}^{1}$-contractible:

\begin{Bsp}[{{\cite[Example 5.5.2.5]{AO}}}]
Let $k = \mathbb{C}$ and $k > l \geq 2$. The tom Dieck-Petrie surface $\{\dfrac{{(xz+1)}^k - {(yz+1)}^l -z}{z} = 0 \} \subset \mathbb{A}^{3}_{\mathbb{C}}$ is a smooth topologically contractible and stably $\mathbb{A}^1$-contractible surface, but has logarithmic Kodaira dimension $1$ and hence cannot be isomorphic to $\mathbb{A}^2_{\mathbb{C}}$. In particular, it is not $\mathbb{A}^1$-contractible by \cite[Theorem 1.1]{CR}.
\end{Bsp}

The relationship between stable $\mathbb{A}^1$-contractibility and topological contractibility is more subtle. First of all, it follows from representability results that stably $\mathbb{A}^1$-contractible smooth affine varieties have the integral motivic cohomology of the base field. One has the criterion for stable $\mathbb{A}^1$-contractibility via Bloch's higher Chow groups below (at least over algebraically closed fields of characteristic $0$).\\
Higher Chow groups were defined by Bloch in \cite{B1} and are denoted $CH_{j}(X,i)$, where $X$ is an equidimensional quasi-projective scheme over a field $k$. One then defines $CH_{\ast}(X,i) = \bigoplus_{j \geq 0} CH_{j}(X,i)$ and $CH_{\ast}(X) = \bigoplus_{i \geq 0} CH_{\ast}(X,i)$. Bloch's higher Chow groups are contravariantly functorial for flat morphisms and covariantly functorial for proper maps (cf. \cite[Proposition 1.3]{B1}), they satisfy homotopy invariance (cf. \cite[Theorem 2.1]{B1}) and admit long exact localization sequences (cf. \cite{B2}). If $X$ has dimension $d$, one sets $CH^j (X,i) = CH_{d-j} (X,i)$. Voevodsky proved in \cite{V1} that there are natural isomorphisms
\begin{center}
$H^{p,q}(X, \mathbb{Z}) \cong CH^q (X,2p-q)$
\end{center}
for any smooth equidimensional quasi-projective scheme $X$ of dimension $d$ over any field $k$. As indicated above, higher Chow groups and therefore essentially motivic cohomology groups give a criterion for the stable $\mathbb{A}^1$-contractibility of smooth affine varieties over algebraically closed fields of characteristic $0$:

\begin{Thm}[{{\cite[Theorem 3.1]{DPO}}}]
Let $X$ be a smooth affine variety over an algebraically closed field $k$ of characteristic $0$ and $x \in X$ be a closed point. If for every affine $Y \in Sm_k$ the projection maps $pr_Y : X \times_k Y \rightarrow Y$ induce isomorphisms on higher Chow groups, then $(X,x)$ is stably $\mathbb{A}^1$-contractible.
\end{Thm}

A lot of examples of topologically contractible smooth affine complex varieties are given by so-called affine modifications; the corresponding result was proven in \cite[Corollary 3.1]{KZ}. It was proven in \cite{DPO} that affine modification are stably $\mathbb{A}^1$-contractible under suitable hypotheses (cf. \cite[Theorem 3.2]{DPO}); we refer the reader to \cite{KZ}, \cite{Z} and \cite{DPO} for detailed discussions of affine modifications. The results on affine modifications imply that many examples of topologically contractible smooth affine varieties are automatically stably $\mathbb{A}^1$-contractible; in fact, a conjecture of Asok-{\O}stv{\ae}r relates the notions of topological contractibility and stably $\mathbb{A}^1$-contractibility as follows:

\begin{Conjecture}[{{\cite[Conjecture 5.5.3.11]{AO}}}]\label{Conjecture}
If $\mathit{X}$ is a topologically contractible smooth affine complex variety and $\mathit{x}$ a chosen closed basepoint, then $(X,x)$ is stably $\mathbb{A}^1$-contractible.
\end{Conjecture}

In dimensions $\leq 2$, this conjecture is known to be true (cf. \cite[Theorem 1]{A}); see the discussion before \cite[Conjecture 5.5.3.11]{AO} and also \cite[Remark 5.5.3.12]{AO} for further information on this conjecture.\\
We now present well-known examples of topologically contractible smooth affine complex varieties called Koras-Russell threefolds. These varieties were studied by Koras-Russell in the context of the solution of the linearization problem for $\mathbb{C}^{\times}$-actions (cf. \cite{KR}).

\begin{Bsp}[{{\cite[(3.1), (3.2)]{HKO}}}]\label{KREX}
Let $k$ be any field of characteristic $0$.\\
(a) Consider the subvarieties of $\mathbb{A}^{4}_k = Spec (k[x,y,z,t])$ defined by the equations
\begin{center}
$x + x^{m}y + z^{\alpha_{2}} + t^{\alpha_{3}}$
\end{center}
with $m, \alpha_{2}, \alpha_{3} \geq 2$ and $\langle \alpha_{2},\alpha_{3} \rangle = \mathbb{Z}$. These varieties are called the Koras-Russell threefolds of the first kind. If $k=\mathbb{C}$, the threefolds are topologically contractible. For any algebraically closed field $k$ of characteristic $0$, it was proven by Hoyois-Krishna-{\O}stv{\ae}r that these varieties equipped with $0$ as a basepoint are stably $\mathbb{A}^1$-contractible (cf. \cite[Theorem 4.2]{HKO}); Dubouloz-Fasel even proved that Koras-Russell threefolds of the first kind are $\mathbb{A}^1$-contractible over any base field $k$ of characteristic $0$ (cf. \cite[Theorem 1]{DF}). In case $k = \mathbb{C}$ Koras-Russell threefolds of the first kind are known to have logarithmic Kodaira dimension $-\infty$ as they have non-trivial $\mathbb{C}_+$-actions (cf. \cite[Exercise 8.7]{Z} and \cite[Introduction]{P}), but they can nevertheless be distinguished from $\mathbb{A}^3_{\mathbb{C}}$ by the so-called Makar-Limanov invariant (cf. \cite{ML}).\\
(b) Secondly, consider the subvarieties defined by the equations
\begin{center}
$x + {(x^{b}+ z^{\alpha_{2}})}^{l} y + t^{\alpha_{3}}$
\end{center}
with $b,l, \alpha_{2}, \alpha_{3} \geq 2$ and $\langle \alpha_{2},b\alpha_{3} \rangle = \mathbb{Z}$. These varieties are called the Koras-Russell threefolds of the second kind and are also topologically contractible whenever $k=\mathbb{C}$. Again, Hoyois-Krishna-{\O}stv{\ae}r proved that these varieties are stably $\mathbb{A}^1$-contractible over any algebraically closed field of characteristic $0$ (\cite[Theorem 4.2]{HKO}). It is an open question whether they are also $\mathbb{A}^1$-contractible or not. As the Koras-Russell threefolds of the first kind, they have logarithmic Kodaira dimension $-\infty$ (cf. \cite[Exercise 8.7]{Z} and \cite[Introduction]{P}) over $\mathbb{C}$, but can also be distinguished from $\mathbb{A}^3_{\mathbb{C}}$ by the so-called Makar-Limanov invariant.
\end{Bsp}

Finally, we remark that there is also a third kind of Koras-Russell threefolds (usually defined over $\mathbb{C}$). While the first two kinds of Koras-Russell threefolds can be described very explicitly as subvarieties of $\mathbb{A}^4_{\mathbb{C}}$ defined by concrete equations as above, it seems that Koras-Russell threefolds of the third kind may not be defined by one specific type of equation; examples of this third kind of Koras-Russell threefolds can be constructed as cyclic coverings of the Koras-Russell threefolds of the first and second kind (see Section \ref{Examples}). The threefolds of the third kind are exactly those Koras-Russell threefolds which admit only trivial $\mathbb{C}^+$-actions (Koras-Russell threefolds of the first and second kind admit non-trivial $\mathbb{C}^+$-actions; see \cite{KML} or \cite[Introduction]{P}). Just like affine modifications, cyclic coverings as defined in \cite[Definition 5.1]{Z} give many examples of topologically contractible smooth affine complex varieties (cf. \cite[Theorem 5.1]{Z}). There exist no analogous results on the $\mathbb{A}^1$-contractibility or stable $\mathbb{A}^1$-contractibility of cyclic coverings.

\subsection{Vector bundles over smooth affine contractible varieties}\label{2.4}

We now discuss results on vector bundles over contractible varieties, i.e., $\mathbb{A}^{1}$-contractible, stably $\mathbb{A}^{1}$-contractible or topologically contractible varieties.\\
First of all, we have already seen that affine spaces are the most basic examples of contractible varieties (in each sense). The Quillen-Suslin theorem shows that vector bundles over affine spaces are trivial (cf. \cite{Q},\cite{S1}):

\begin{Thm}
For any $n \geq 0$ and field $k$, algebraic vector bundles over affine spaces $\mathbb{A}^{n}_{k}$ are trivial.
\end{Thm}

This is known as Serre's problem or Serre's question. In the context of $\mathbb{A}^{1}$-homotopy theory, there is a representability result which is analogous to the Pontryagin-Steenrod representability result for vector bundles in topology (cf. \cite{AHW}):

\begin{Thm}[{{\cite[Theorem 5.2.3]{AHW}}}]
Let $k$ be a field (or a ring which is smooth over a Dedekind ring with perfect residue fields or a ring for which the Bass-Quillen conjecture holds) and $n \geq 0$. For any smooth affine scheme $X$ over $k$, there are natural bijections
\begin{center}
$V_{n}(X) = [X, BGL_{n}]_{\mathcal{H}(k)}$,
\end{center}
where $BGL_{n}$ is the simplicial classifying space associated to the group scheme $GL_{n}$ of invertible matrices of rank $n$.
\end{Thm}

By analogy with the situation in topology, there is an obstruction theory involving a Postnikov tower for $BGL_{n}$. This was used in order to prove many classification results on vector bundles over affine schemes (cf. \cite{AF1}, \cite{AF2}). As a direct consequence of the representability result above it follows that algebraic vector bundles over any $\mathbb{A}^1$-contractible smooth affine scheme $X$ are trivial. The generalized Serre question asks whether algebraic vector bundles over any topologically contractible smooth affine complex variety $X$ are always trivial or not (cf. \cite[Question 6]{AO}). Due to Serre's splitting theorem (cf. \cite[Th{\'e}or{\`e}me 1]{S}), it suffices to study vector bundles of rank $\leq \dim (X)$. Let us now state some results on the generalized Serre question:

\begin{Thm}[{{\cite[Theorem 1]{G}}}]
If $X$ is a topologically contractible smooth affine complex variety, then $CH^{1}(X) = 0$.
\end{Thm}

In other words, line bundles on topologically contractible varieties are always trivial. As a direct corollary, one immediately obtains a positive answer to the generalized Serre question in dimension $1$ by using Serre's splitting theorem (cf. \cite[Th{\'e}or{\`e}me 1]{S}):

\begin{Kor}
If $X$ is a topologically contractible smooth affine complex variety of dimension $d=1$, then all algebraic vector bundles on $X$ are trivial.
\end{Kor}

In dimension $2$, vector bundles can be completely classified by their Chern classes:

\begin{Thm}
If $X$ is a smooth affine variety of dimension $d=2$ over an algebraically closed field $k$, then the natural map $({c}_{1},{c}_{2}): {V}_{2}(X) \xrightarrow{\cong} CH^{1}(X) \times CH^{2}(X)$ is a bijection.
\end{Thm}

The corresponding vanishing statement for topologically contractible surfaces was proven by Gurjar-Shastri in 1989:

\begin{Thm}
If $X$ is a topologically contractible smooth affine complex variety of dimension $d=2$, then $CH^{2}(X) = 0$. In particular, all algebraic vector bundles on $X$ are trivial.
\end{Thm}

\begin{proof}
See \cite{GS1} and \cite{GS2}.
\end{proof}

Increasing the dimension, one always has to prove more classification results and consequently compute more cohomology groups:

\begin{Thm}[{{\cite[Theorem 2.1(iii)]{KM}}}]
If $X$ is a smooth affine variety of dimension $d=3$ over an algebraically closed field $k$ of characteristic $\neq 2$, then the natural map $({c}_{1},{c}_{2},{c}_{3}): {V}_{3}(X) \xrightarrow{\cong} CH^{1}(X) \times CH^{2}(X) \times CH^{3}(X)$ is a bijection.
\end{Thm}

\begin{Thm}[{{\cite[Theorem 1]{AF1}}}]
If $X$ is a smooth affine variety of dimension $d=3$ over an algebraically closed field $k$ of characteristic $\neq 2$, then the natural map $({c}_{1},{c}_{2}): {V}_{2}(X) \xrightarrow{\cong} CH^{1}(X) \times CH^{2}(X)$ is a bijection.
\end{Thm}

\begin{Kor}
If $X$ is a topologically contractible smooth affine complex variety of dimension $d=3$, then all algebraic vector bundles on $X$ are trivial if and only if $CH^{2}(X)$ and $CH^{3}(X)$ are trivial.
\end{Kor}

Unfortunately, in higher dimensions, it follows from the construction of stably trivial non-trivial vector bundles given in \cite{NMK} that one cannot classify vector bundles just by considering Chern classes in Chow groups. Nevertheless, one can still hope to get some computable cohomological obstructions for the triviality of all vector bundles:

\begin{Thm}[{{\cite[Corollary 3.21]{Sy}}}]\label{Syed}
If $X$ is a smooth affine variety of dimension $d=4$ over an algebraically closed field $k$ with $6 \in k^{\times}$, all algebraic vector bundles on $X$ are trivial if $CH^{\mathit{i}}(X)=0$ for $1 \leq \mathit{i} \leq 4$ and $H_{Nis}^{2}(X, \textbf{\textrm{I}}^{3})=0$.
\end{Thm}

\begin{Kor}
If $X$ is a topologically contractible smooth affine complex variety of dimension $d=4$, then all algebraic vector bundles on $X$ are trivial if $CH^{\mathit{i}}(X)=0$ for $\mathit{i}=2,3,4$ and $H_{Nis}^{2}(X, \textbf{\textrm{I}}^{3})=0$.
\end{Kor}

Note that vector bundles on stably $\mathbb{A}^{1}$-contractible varieties of low dimension are therefore trivial:

\begin{Kor}
If $X$ is a stably $\mathbb{A}^1$-contractible smooth affine variety of dimension $d \leq 4$ over an algebraically closed field with $6 \in k^{\times}$, then all algebraic vector bundles on $X$ are trivial.
\end{Kor}

For topologically contractible smooth affine complex varieties of dimension $\leq 4$, it remains to determine whether all the obstructions to the triviality of all vector bundles are vanishing or not. In particular, it is of interest whether Chow groups of topologically contractible smooth affine complex varieties are trivial or not. We conclude this section with the following observation:

\begin{Prop}\label{Euler}
Let $X$ be a smooth affine variety of dimension $d \geq 3$ over an algebraically closed field $k$ of characteristic $0$. Assume that $CH^d (X) \otimes_{\mathbb{Z}} \mathbb{Q} = 0$. Then every vector bundle of rank $d$ has trivial direct summand of rank $1$.
\end{Prop}

\begin{proof}
A vector bundle of rank $d$ has a trivial direct summand of rank $1$ if and only if its top Chern class in $CH^d (X)$ is trivial (cf. \cite[Theorem 3.8]{M}). It is well-known that $CH^d (X)$ is a uniquely divisible abelian group (cf. \cite{Sr}) and hence a $\mathbb{Q}$-vector space. So there is a $\mathbb{Q}$-linear isomorphism of the form $CH^d (X) \cong \mathbb{Q}^I$ for some set $I$. By assumption, $CH^d (X) \otimes_{\mathbb{Z}} \mathbb{Q} = 0$. But, as $\mathbb{Q} \otimes_{\mathbb{Z}} \mathbb{Q} \cong \mathbb{Q}$, it follows that $0 = CH^d (X) \otimes_{\mathbb{Z}} \mathbb{Q} \cong \mathbb{Q}^I \otimes_{\mathbb{Z}} \mathbb{Q} \cong \mathbb{Q}^I \cong CH^d (X)$, which finishes the proof.
\end{proof}

\section{Cyclic coverings and their motivic cohomology}\label{3}

In this section we discuss the motivic cohomology of cyclic coverings. As mentioned above, cyclic coverings give many examples of topologically contractible smooth affine complex varieties; we refer the reader to \cite[Section 5]{Z} for a nice survey of related results. It follows from the previous discussion that stably $\mathbb{A}^1$-contractible or $\mathbb{A}^1$-contractible smooth affine varieties have the integral motivic cohomology of the base field. In view of Conjecture \ref{Conjecture} it is therefore of interest to study the motivic cohomology of cyclic coverings. We will be focusing on cyclic coverings $\varphi_{s}: Y_{s} \rightarrow X$ as defined in Definition \ref{CyclicCoverings}. Troughout this section, we will use the notation from Definition \ref{CyclicCoverings} and work under the General Assumptions \ref{Assumptions}.

\begin{Rem}
Note that if $\mathcal{O}_{X}(X)$ is a UFD, then Eisenstein's criterion for irreducibility implies that the second assumption in the General Assumptions \ref{Assumptions} holds as soons as $f$ is a prime element in $\mathcal{O}_{X}(X)$.
\end{Rem}

Consider the commutative diagram (of open/closed immersions):

\begin{center}
$\begin{xy}
  \xymatrix{
     Y_s \setminus F \ar[d]_{\varphi_{s}} \ar[r]^{i} & Y_s \ar[d]_{\varphi_{s}}  & F \ar[d]_{\varphi_{s}}\ar[l] \\
     X \setminus F_0 \ar[r]^{i} & X & F_0 \ar[l]}
\end{xy}$
\end{center}

Let $A = \mathcal{O}_X (X)$. Then $Y_s \setminus \varphi_s^{-1}(F_0) = \varphi_s^{-1} (X \setminus F_0) = Y_s \setminus F$ and the morphism on the left-hand side is given by the ring homomorphism
\begin{center}
$A_{f} \rightarrow A_{f}[u]/\langle f - u^s \rangle = A[u]_{u}/\langle f - u^s \rangle$.
\end{center}
This is a standard {\'e}tale ring homomorphism (cf. \cite[\href{https://stacks.math.columbia.edu/tag/00UB}{Tag 00UB}]{stacks-project}); moreover, it is clearly finite (it is an integral ring extension). In particular, the corresponding morphism of schemes is a finite {\'e}tale cover. The induced extension of function fields
\begin{center}
$k(X) \rightarrow k(X)[u]/\langle f - u^s \rangle = k(Y_s)$
\end{center}
is finite Galois and its Galois group is simply given by the group of $s$th roots of unity $\mu_{s} (k) = \{ x \in k~|~a^s =1\} \cong \mathbb{Z}/s\mathbb{Z}$. Altogether,
\begin{center}
$\varphi_s: Y_{s}\setminus F \rightarrow X \setminus F_{0}$
\end{center}
is a finite {\'e}tale Galois cover (cf. \cite[Definition 6.1 and subsequent remarks]{Mi}). Note that there are induced actions of the Galois group $\mu_s (k)$ on $Y_s$ and $Y_s \setminus F$; we call these actions the Galois actions of $\mu_s (k)$ and usually denote them by $\ast$. One checks easily that $X$ (resp. $X \setminus F_0$) is the quotient of $Y_s$ (resp. $Y_s \setminus F$) under the Galois action (but we do not use this fact in this paper). Finally, note that $\varphi_s$ restricts to an isomorphism $F \xrightarrow{\cong} F_0$.\\
We will now prove several statements about finite {\'e}tale Galois covers which, in particular, apply to the morphism $\varphi_{s}: Y_{s}\setminus F \rightarrow X \setminus F_{0}$ above. For this purpose, we fix a finite {\'e}tale Galois cover $\varphi: U \rightarrow V$ between smooth $k$-varieties with Galois group $G$ and of degree $s = |G|$. We will call the corresponding action of the group $G$ on $U$ the Galois action and denote it usually by $\ast$.\\
Recall from Section \ref{Motivic cohomology} that we denote by $Cor_k$ the category of finite correspondences over $k$ (cf. \cite[Lecture 1]{MVW}) and, for any scheme $X$ of finite type over $k$, by $\mathbb{Z}_{tr}(X)$ the representable presheaf with transfers given by $Y \mapsto Cor_k (Y,X)$ (cf. \cite[Definition 2.8 and Exercise 2.11]{MVW}); analogously, for any commutative ring $R$, we denote by $R_{tr}(X)$ the representable presheaf with transfers of $R$-modules, i.e., $R_{tr}(X)(Y) = \mathbb{Z}_{tr}(X)(Y) \otimes_{\mathbb{Z}} R$. We first prove the following lemma on Galois morphisms:

\begin{Lem}\label{Etale}
Let $\varphi:U \rightarrow V$ be a finite {\'e}tale Galois cover between smooth $k$-varieties with Galois group $G$. Then $\varphi$ induces a bijection $Cor_{k} (V,W) \cong {Cor_{k}(U,W)}^{G}$ for any $W \in Sm_k$.
\end{Lem}

\begin{proof}
%This is claimed in \cite[Exercise 6.5]{MVW}.
The lemma follows from the fact that the presheaf with transfers $\mathbb{Z}_{tr}(W)$ is a sheaf in the {\'e}tale topology (cf. \cite[Lemma 6.2]{MVW}): Therefore one has an exact sequence
\begin{center}
$0 \rightarrow {\mathbb{Z}}_{tr}(W)(V) \xrightarrow{\varphi^{\ast}} {\mathbb{Z}}_{tr}(W)(U) \xrightarrow{\pi_{1}^{\ast}-\pi_{2}^{\ast}} {\mathbb{Z}}_{tr}(W)(U \times_V U)$.
\end{center}
But for a finite {\'e}tale Galois cover as $\varphi$ the induced morphism $U \times G \rightarrow U \times_V U, (u,g) \mapsto (u, u \ast g)$, where $U \times G$ denotes the disjoint union $\sqcup_{g \in G} U_g$ of copies of $U$, is an isomorphism. So, up to this isomorphism, one obtains an exact sequence
\begin{center}
$0 \rightarrow {\mathbb{Z}}_{tr}(W)(V) \xrightarrow{\varphi^{\ast}} {\mathbb{Z}}_{tr}(W)(U) \xrightarrow{{(id^{\ast}-g^{\ast})}_{g \in G}} \bigoplus_{g \in G} {\mathbb{Z}}_{tr}(W)(U_g)$,
\end{center}
which identifies $Cor (V,W)$ precisely with ${Cor(U,W)}^{G}$.
\end{proof}

For $g \in G$, we will denote the automorphism $y \mapsto y \ast g$ of $U$ just by $g$.

\begin{Kor}\label{Existence}
There is a finite correspondence $\mu : V \rightarrow U$ such that $\mu \circ \varphi = \sigma := \sum_{g \in G} g$. Furthermore, one has $\varphi \circ \mu = s \cdot id_{X \setminus F_{0}}$.
\end{Kor}

\begin{proof}
The first statement follows from the previous lemma for $\varphi$ and $W = U$. For the second statement, first apply the previous lemma to $\varphi$ and $W = V$ in order to get the identification $Cor_{k} (V,V) \cong {Cor_{k}(U,V)}^{G}$ induced by $\varphi$. So it follows in particular that it suffices to show that $\varphi \circ \mu \circ \varphi = s \cdot \varphi$ for the second statement; but this follows directly from the first statement.
\end{proof}

\begin{Rem}\label{graph}
It is possible to find an explicit finite correspondence $\mu$ as in the proof of the preceeding corollary. In fact, let $U,V$ smooth affine $k$-varieties of the same dimension and $\varphi: U \rightarrow V$ be a finite surjective morphism of degree $d \geq 1$. Then consider the finite correspondence $\mu$ from $V$ to $U$ represented by the transpose $\Gamma_{\varphi}^{t}$ of the graph of $\varphi$. Then the intersection product $(\Gamma_{\varphi}^{t} \times_k V) \cdot (V \times_k \Gamma_{\varphi})$ is just given by the intersection $T$ of $\Gamma_{\varphi}^{t} \times_k V$ and $V \times_k \Gamma_{\varphi}$ in $V \times_k U \times_k V$, which is isomorphic to $U$ and hence irreducible (and has intersection multiplicity $1$ by \cite[\href{https://stacks.math.columbia.edu/tag/0B1I}{Tag 0B1I}]{stacks-project}). The image of this intersection $T$ under the morphism $p: V \times_k U \times_k V \rightarrow V \times_k V$ is isomorphic to the diagonal and hence isomorphic to $V$; the restriction of $p$ to $T$ then corresponds precisely to $\varphi$. It follows then from the definitions of the pushforward and of the composition law in $Cor_k$ that $\varphi \circ \mu = d \cdot id_{V}$ in $Cor_k$.\\
Secondly, assume now that $U,V$ are smooth affine $k$-varieties of the same dimension and $\varphi: U \rightarrow V$ is a finite {\'e}tale Galois cover of degree $d \geq 1$ with Galois group $G$. We claim that $\mu \circ \varphi = \sum_{g \in G} g$ in $Cor_k$, where $\mu$ is defined as in the previous paragraph: Note that the intersection $S$ of $\Gamma_{\varphi} \times_k U$ and $U \times_k \Gamma_{\varphi}^{t}$ in $U \times_k V \times_k U$ is canonically isomorphic to $U \times_V U$ and hence canonically isomorphic to the disjoint union $U \times G = \sqcup_{g \in G} U_g$ of copies of $U$ via the isomorphism $U \times G \rightarrow U \times_V U, (u,g) \mapsto (u, u \ast g)$ (as in the proof of Lemma \ref{Etale}). It follows that the intersection product $(\Gamma_{\varphi} \times_k U) \cdot (U \times_k \Gamma_{\varphi}^{t})$ is just given by the sum of these irreducible components and the pushforward along $p: U \times_k V \times_k U \rightarrow U \times_k U$ is just the sum $\sum_{g \in G} \Gamma_{g}$ of the graphs of all the automorphism of $Y$ in $G$. As any graph $\Gamma_g$ corresponds exactly to the automorphism $g$ in $Cor_k$, this proves the claim.\\
In conclusion, as $\varphi_{s}: Y_{s} \setminus F \rightarrow X \setminus F$ is a finite {\'e}tale Galois cover with Galois group $\mu_s (k)$, it follows from the previous two paragraphs that one can define $\mu$ as the finite correspondence from $X \setminus F_{0}$ to $Y_{s} \setminus F$ represented by the transpose $\Gamma_{\varphi_{s}}^{t}$ of the graph of $\varphi_s$. Then this finite correspondence satisfies the properties of $\mu$ in the previous corollary. In fact, as $\varphi_{s}: Y_{s} \rightarrow X$ is at least also finite, the first paragraph above shows that the finite correspondence $\overline{\mu}$ from $X$ to $Y_{s}$ represented by the transpose $\Gamma_{\varphi_{s}}^{t}$ of the graph of $\varphi_s$ satisfies $\varphi_s \circ \overline{\mu} = s \cdot id_{X}$.
\end{Rem}

Recall that, for $Y \in Sm_k$, we denote by $H^{p,q}(Y,R)$ its motivic cohomology in bidegree $(p,q) \in \mathbb{Z}^2$ with coefficients in a commutative ring $R$. If there is a finite group $G$ acting on $Y$, then we denote by ${H^{p,q}(Y,R)}^{G}$ the subgroup of elements in $H^{p,q}(Y,R)$ fixed by the action of $G$.

\begin{Kor}\label{FixedSubgroup}
Let $(p,q) \in \mathbb{Z}^2$. If $s \in R^{\times}$, then $\varphi$ induces isomorphisms ${H^{p,q}(U,R)}^{G} \cong H^{p,q}(V,R)$. The analogous statement holds for any presheaf of $R$-modules with transfers.
\end{Kor}

\begin{proof}
The statement follows formally from the Corollary \ref{Existence} and representability of motivic cohomology in $\textbf{DM}_{Nis}^{eff,-}(k,R)$ (cf. \cite[Proposition 14.16]{MVW}), which implies that motivic cohomology groups are presheaves of $R$-modules with transfers: Note that ${\varphi}^{\ast}$ automatically maps into the group ${H^{p,q}(U,R)}^{G}$ as $\varphi \circ g = \varphi$ and hence $g^{\ast} \circ {\varphi}^{\ast} = {\varphi}^{\ast}$. If we let $\mu: V \rightarrow U$ be as in Corollary \ref{Existence}, then the composite
\begin{center}
$H^{p,q}(V,R) \xrightarrow{\varphi^{\ast}} {H^{p,q}(U,R)}^{G} \hookrightarrow H^{p,q}(U,R) \xrightarrow{\mu^{\ast}} H^{p,q}(V,R)$
\end{center}
is multiplication by $s$ and hence, as $s \in R^{\times}$, an isomorphism; in particular, $\varphi^{\ast}: H^{p,q}(V,R) \rightarrow {H^{p,q}(U,R)}^{G}$ is injective. Similarly, the composite
\begin{center}
${H^{p,q}(U,R)}^{G} \hookrightarrow H^{p,q}(U,R) \xrightarrow{\mu^{\ast}} H^{p,q}(V,R) \xrightarrow{\varphi^{\ast}} {H^{p,q}(U,R)}^{G}$
\end{center}
is $\sum_{g \in G} g^{\ast}$, which is multiplication by $s$ on ${H^{p,q}(U,R)}^{G}$ and hence an isomorphism as well because $s \in R^{\times}$; in particular $\varphi^{\ast}$ is also surjective and hence an isomorphism. The reasoning for any presheaf of $R$-modules with transfers is completely analogous.
\end{proof}

\begin{Kor}
Let $(p,q) \in \mathbb{Z}^2$. If $s \in R^{\times}$, then there is a split short exact sequence of $R$-modules of the form
\begin{center}
$0 \rightarrow \ker(\mu^{\ast}) \rightarrow {H^{p,q}(U,R)} \xrightarrow{\mu^{\ast}} H^{p,q}(V,R) \rightarrow 0$
\end{center}
and, in particular, an isomorphism ${H^{p,q}(U,R)} \cong H^{p,q}(V,R) \oplus \ker (\mu^{\ast})$. The analogous statements hold for any presheaf of $R$-modules with transfers.
\end{Kor}

We say that the Galois action of $G$ on $U$ is cohomologically trivial on $H^{p,q}(U, R)$ if ${H^{p,q}(U, R)}^{G} = {H^{p,q}(U, R)}$, i.e., any $g \in G$ acts as identity on the group $H^{p,q}(U,R)$. Moreover, we say that the Galois action is motivically trivial if $g$ is the identity on $R_{tr}(U)$ in $\textbf{DM}_{Nis}^{eff,-}(k,R)$ for all $g \in G$. By representability of motivic cohomology in $\textbf{DM}_{Nis}^{eff,-}(k,R)$ (cf. \cite[Proposition 14.16]{MVW}), it follows that motivic triviality is stronger than cohomological triviality. The following statement is a direct consequence of Corollary \ref{FixedSubgroup}:

\begin{Kor}\label{Iso1}
Let $(p,q) \in \mathbb{Z}^2$ and assume $s \in R^{\times}$. Then the Galois action of $G$ on $U$ is cohomologically trivial on $H^{p,q}(U, R)$ if and only if $\varphi$ induces isomorphisms ${H^{p,q}(U,R)} \cong H^{p,q}(V,R)$. The analogous statement holds for any presheaf of $R$-modules with transfers.
\end{Kor}

The following statement is a consequence of Corollary \ref{Existence}:

\begin{Kor}
If $s \in R^{\times}$ and the Galois action of $G$ on $U$ is motivically trivial, then $\varphi$ induces isomorphisms $R_{tr}(U) \cong R_{tr}(V)$ in $\textbf{DM}_{Nis}^{eff,-}(k,R)$.
\end{Kor}

The following corollary is a consequence of Corollary \ref{Iso1}:

\begin{Kor}
Assume $s \in R^{\times}$ and let $(p,q) \in \mathbb{Z}^2$. Then $U \rightarrow Spec(k)$ induces an isomorphism on the motivic cohomology groups with coefficients in $R$ in bidegree $(p,q)$ if and only if $V \rightarrow Spec(k)$ induces an isomorphism on the motivic cohomology groups with coefficients in $R$ in bidegree $(p,q)$ and the Galois action of $G$ on $U$ is cohomologically trivial on $H^{p,q}(U,R)$.
\end{Kor}

\begin{Rem}\label{AlwaysInjective}
Let $R$ be a commutative ring with $s \in R^{\times}$ and let $(p,q) \in \mathbb{Z}^2$. Then it follows in particular from the previous corollaries that the structure morphism $X \setminus F_{0} \rightarrow Spec(k)$ induces an isomorphism on motivic cohomology with coefficients in $R$ in bidegree $(p,q)$ whenever the structure morphism $Y_{s} \setminus F \rightarrow Spec(k)$ does. By using Remark \ref{graph} and, in particular, the finite correspondence $\overline{\mu}$, it follows also that the structure morphism $X \rightarrow Spec(k)$ induces an isomorphism on motivic cohomology with coefficients in $R$ in bidegree $(p,q)$ whenever the structure morphism $Y_{s} \rightarrow Spec(k)$ does. In fact, it follows from Remark \ref{graph} that $\varphi_{s}^{\ast}: H^{p,q}(X,R) \rightarrow H^{p,q}(Y_{s},R)$ is always injective.
\end{Rem}

%\begin{Thm}\label{cdh-theorem}
%Let $(p,q) \in \mathbb{Z}^2$ and assume that $s \in R^{\times}$. Then the Galois action of $\mu_s (k)$ on $H^{p,q}(Y_{s}\setminus F, R)$ is trivial if and only if the morphism $\varphi_{s}: Y_{s} \rightarrow X$ induces an isomorphism $H^{p,q}(X,R)\cong H^{p,q}(Y_{s},R)$.
%\end{Thm}

%\begin{proof}
%The diagram
%\begin{center}
%$\begin{xy}
%  \xymatrix{
%     Y_s \setminus F \ar[d]_{\varphi_{s}} \ar[r]^{i} & Y_s \ar[d]^{\varphi_{s}}\\
%     X \setminus F_0 \ar[r]_{i} & X}
%\end{xy}$
%\end{center}
%is a cdh distinguished square (cf. \cite[Chapter 6, \S 1.1.12]{BCDFO}). In particular, since motivic cohomology satisfies cdh descent (cf. \cite[Chapter 6, Proposition 2.4.1]{BCDFO}), there is a long exact sequence of motivic cohomology groups of the form
%\begin{center}
%$... \rightarrow H^{p,q}(X,R) \xrightarrow{(i^{\ast},\varphi_{s}^{\ast})} H^{p,q}(X \setminus F_{0},R) \oplus H^{p,q}(Y_{s},R) \xrightarrow{\varphi^{\ast}_{s} - i^{\ast}} H^{p,q}(Y_{s} \setminus F,R) \rightarrow H^{p+1,q}(X,R) \rightarrow ...$
%\end{center}
%Now recall from Remark \ref{graph} that both maps $\varphi_s : Y_{s}\setminus F \rightarrow X \setminus F_{0}$ and $\varphi_s : Y_{s} \rightarrow X$ induce injective maps on motivic cohomology with coefficients in $R$. Using the sequence above, it is straightforward to check that $\varphi_s : Y_{s}\setminus F \rightarrow X \setminus F_{0}$ induces an isomorphism on motivic cohomology with coefficients in $R$ in bidegree $(p,q)$ if and only if $\varphi_{s}: Y_{s} \rightarrow X$ does so in bidegree $(p,q)$.
%\end{proof}

We now focus on the morphism $\varphi_{s}: Y_{s} \setminus F \rightarrow X \setminus F_{0}$. In view of Corollary \ref{Iso1}, it is now of interest to investigate under which circumstances the Galois action of $\mu_s (k)$ on $H^{p,q}(Y_{s}\setminus F,R)$ is indeed trivial. We will say that an action of $\mu_s (k)$ on $Y_s \setminus F$ extends to a morphism $h:\mathbb{A}^1_k \times_k (Y_s \setminus F) \rightarrow Y_s \setminus F$ (resp. $h:\mathbb{G}_{m,k} \times_k (Y_s \setminus F) \rightarrow Y_s \setminus F$) if there is a morphism $h:\mathbb{A}^1_k \times_k (Y_s \setminus F) \rightarrow Y_s \setminus F$ (resp. $h:\mathbb{G}_{m,k} \times_k (Y_s \setminus F) \rightarrow Y_s \setminus F$) which recovers the given action of $\mu_s (k)$ on $Y_{s} \setminus F$ when evaluated at the $k$-points of $\mathbb{A}^1_k$ (resp. $\mathbb{G}_{m,k}$) corresponding to the elements of $\mu_s (k)$.

\begin{Lem}\label{Extension}
An action of $\mu_s (k)$ on $Y_s \setminus F$ is motivically trivial (and hence cohomologically trivial) if it extends to a morphism $h:\mathbb{A}^1_k \times_k (Y_s \setminus F) \rightarrow Y_s \setminus F$.
\end{Lem}

\begin{proof}
For any $g \in \mu_s (k) \subset \mathbb{A}^1_k (k)$, we have by assumption a factorization
\begin{center}
$g: (Y_s \setminus F)= k \times_k (Y_s \setminus F) \xrightarrow{(g,id)} \mathbb{A}^1_k \times_k (Y_s \setminus F) \xrightarrow{h} Y_s \setminus F$.
\end{center}
Since any $(g,id)$ defines a section to the projection of $\mathbb{A}^1_k \times_k (Y_s \setminus F)$ onto $Y_s \setminus F$ and this projection is an isomorphism in $\textbf{DM}_{Nis}^{eff,-}(k,R)$, it follows that ${(g,id)} = {(1,id)}$ in $\textbf{DM}_{Nis}^{eff,-}(k,R)$ for any $g \in \mu_s (k)$. Hence $g$ equals the identity in $\textbf{DM}_{Nis}^{eff,-}(k,R)$ for all $g \in \mu_s (k)$.
\end{proof}

Regarding motivic cohomology groups, the reasoning in the proof of Lemma \ref{Extension} also works for $\mathbb{G}_{m,k}$ with an additional hypothesis:

\begin{Lem}\label{Extension-2}
Let $(p,q) \in \mathbb{Z}^2$. An action of $\mu_s (k)$ on $Y_s \setminus F$ is cohomologically trivial on $H^{p,q}(Y_{s} \setminus F,R)$ if it extends to a morphism $h:\mathbb{G}_{m,k} \times_k (Y_s \setminus F) \rightarrow Y_s \setminus F$ and the projection $\pi: \mathbb{G}_{m,k} \times_k (Y_s \setminus F) \rightarrow (Y_s \setminus F)$ induces an isomorphism $\pi^{\ast}: H^{p,q}(Y_{s} \setminus F,R) \xrightarrow{\cong} H^{p,q}(\mathbb{G}_{m,k} \times_k Y_{s} \setminus F,R)$.
\end{Lem}

\begin{Rem}\label{GM-Chow}
Since the composite
\begin{center}
$Y_{s}\setminus F \xrightarrow{(id,1)} (Y_{s}\setminus F) \times_k \mathbb{G}_{m,k} \xrightarrow{\pi} Y_{s}\setminus F$
\end{center}
induces the identity and hence an isomorphism on motivic cohomology groups, $\pi^{\ast}$ in Lemma \ref{Extension-2} is always injective. In fact, note that the projection map $\pi: \mathbb{G}_{m,k} \times_k (Y_s \setminus F) \rightarrow (Y_s \setminus F)$ always induces an isomorphism on Chow groups (tensored with any coefficient ring $R$). Indeed, the exact localization sequence
\begin{center}
$CH^{i-1} (Y_{s}\setminus F)\otimes_{\mathbb{Z}}R \rightarrow CH^{i}((Y_{s}\setminus F)\times_{k} \mathbb{A}^{1}_{k})\otimes_{\mathbb{Z}}R \rightarrow CH^{i}((Y_{s}\setminus F)\times_{k} \mathbb{G}_{m,k})\otimes_{\mathbb{Z}}R \rightarrow 0$
\end{center}
and homotopy invariance imply that $\pi^{\ast}: CH^{i}(Y_{s}\setminus F) \rightarrow CH^{i}((Y_{s}\setminus F)\times_{k} \mathbb{G}_{m,k})$ is surjective and hence bijective.
\end{Rem}

\begin{Lem}\label{Extension-Lem}
Let $(p,q) \in \mathbb{Z}^2$. Assume that there is a $\mathbb{G}_{m,k}$-action $\gamma$ on $X$ which makes $f$ a quasi-invariant of weight $d \in \mathbb{Z}$ with respect to $\gamma$ and $\langle d,s \rangle = \mathbb{Z}$. Then the Galois action of $\mu_s (k)$ on $Y_s \setminus F$ is cohomologically trivial on $H^{p,q}(Y_{s} \setminus F,R)$ as soon as the projection $\pi: \mathbb{G}_{m,k} \times_k (Y_s \setminus F) \rightarrow (Y_s \setminus F)$ induces an isomorphism $\pi^{\ast}: H^{p,q}(Y_{s} \setminus F,R) \xrightarrow{\cong} H^{p,q}(\mathbb{G}_{m,k} \times_k Y_{s} \setminus F,R)$.
\end{Lem}

\begin{proof}
By assumption, there is a $\mathbb{G}_{m,k}$-action
\begin{center}
$\gamma: \mathbb{G}_{m,k} \times_k X \rightarrow X, (\lambda, x) \mapsto \lambda \ast_{\gamma} x$
\end{center}
on $X$ which makes $f$ a quasi-invariant of weight $d \in \mathbb{Z}$ with respect to $\gamma$. As explained after \cite[Definition 5.1]{Z}, one obtains a new $\mathbb{G}_{m,k}$-action $\tilde{\gamma}$ on $Y_s$ by setting
\begin{center}
$\lambda \ast_{\tilde{\gamma}} (x,u) := (\lambda^s \ast_{\gamma} x, \lambda^d u)$.
\end{center}
More precisely, let
\begin{center}
${()}^{l}: \mathbb{G}_{m,k} \rightarrow \mathbb{G}_{m,k}$
\end{center}
denote $l$th power morphism on $\mathbb{G}_{m,k}$ for $l \in \mathbb{Z}$. Then we first obtain a new $\mathbb{G}_{m,k}$-action $\gamma_{1}$ on $X$ via the composite
\begin{center}
$\gamma_{1}:\mathbb{G}_{m,k} \times_k X \xrightarrow{{()}^s \times_k id_X} \mathbb{G}_{m,k} \times_k X \xrightarrow{\gamma} X$.
\end{center}
Similarly, if we consider the multiplication action of $\mathbb{G}_{m,k}$ on $\mathbb{A}^1_k$ (with coordinate $u$)
\begin{center}
$m: \mathbb{G}_{m,k} \times_k \mathbb{A}^1_k \rightarrow \mathbb{A}^1_{k}, (\lambda, u) \mapsto \lambda \cdot u$,
\end{center}
we obtain a new $\mathbb{G}_{m,k}$-action $\gamma_{2}$ on $\mathbb{A}^1_k$ via the composite
\begin{center}
$\gamma_{2}: \mathbb{G}_{m,k} \times_k \mathbb{A}^1_k \xrightarrow{{()}^{d} \times_k id_{\mathbb{A}^1_k}} \mathbb{G}_{m,k} \times_k \mathbb{A}^1_k \xrightarrow{m} \mathbb{A}^1_k$.
\end{center}
If we let $\pi_{1}: \mathbb{G}_{m,k} \times_k X \times_k \mathbb{A}^1_k \rightarrow \mathbb{G}_{m,k} \times_k X$ and $\pi_{2}: \mathbb{G}_{m,k} \times_k X \times_k \mathbb{A}^1_k \rightarrow \mathbb{G}_{m,k} \times_k \mathbb{A}^1_k$ denote the projections, we then obtain an action $\tilde{\gamma}$ on $X \times_k \mathbb{A}^1_k$ which is simply defined as the product
\begin{center}
$\mathbb{G}_{m,k} \times_k X \times_k \mathbb{A}^1_k \xrightarrow{(\gamma_{1} \circ \pi_{1}, \gamma_{2} \circ \pi_{2})} X \times_k \mathbb{A}^1_k$.
\end{center}
By construction, the action $\tilde{\gamma}$ makes the function $f - u^s \in \mathcal{O}_{X \times_{k} \mathbb{A}^1_k}(X \times_{k} \mathbb{A}^1_k)= \mathcal{O}_{X}(X)[u]$ a quasi-invariant of weight $ds \in \mathbb{Z}$; in particular, $\tilde{\gamma}$ descends to a $\mathbb{G}_{m,k}$-action on $Y_{s}$. By abuse of notation, we also call this action $\tilde{\gamma}$ and this is precisely the action on $Y_{s}$ indicated in the first paragraph of this proof. The action $\tilde{\gamma}$ on $Y_{s}$ clearly restricts to an action on $Y_s \setminus F$ (which, by abuse of notation, we also call $\tilde{\gamma}$).\\
Now we use the assumption that $\langle d,s \rangle = \mathbb{Z}$. Note that when the restriction of $\tilde{\gamma}$ to $\mu_{s}$ is evaluated at $k$-points, i.e., actual $s$th roots of unity over $k$, then one recovers the $d$-th powers of the Galois actions of $\mu_{s}(k)$ on $X$. Since $\langle d,s \rangle = \mathbb{Z}$, taking the $d$-th power defines an automorphism of $\mu_{s}$ and we need not distinguish between the evaluation of $\tilde{\gamma}$ at $k$-points of $\mu_{s}$ and the Galois action of $\mu_s (k)$ when considering the groups ${H^{p,q}(Y_s \setminus F, R)}^{\mu_s (k)}$. Now the reasoning from the proof of Lemma \ref{Extension-2} (or rather Lemma \ref{Extension}) finishes the proof.
\end{proof}

\begin{Rem}\label{Extension-Bicyclic}
The $\mathbb{G}_{m,k}$-action $\tilde{\gamma}$ in the proof of Lemma \ref{Extension-Lem} above also shows: If there is a $\mathbb{G}_{m,k}$-action $\gamma$ on $X$ which makes $f$ a quasi-invariant of weight $d \in \mathbb{Z}$ with respect to $\gamma$, then there is also a $\mathbb{G}_{m,k}$-action $\tilde{\gamma}$ on $X \times_k \mathbb{A}^1_k$ which makes $f - u^s$ (where $u$ is the coordinate of $\mathbb{A}^1_k$) a quasi-invariant of weight $ds \in \mathbb{Z}$ with respect to $\tilde{\gamma}$.
\end{Rem}

The next theorem shows that the triviality of the Galois action of $\mu_s (k)$ on the group $H^{p,q}(Y_{s}\setminus F,R)$ also determines the bijectivity of the homomorphism $\varphi_s^{\ast}: H^{p,q}(X,R) \rightarrow H^{p,q}(Y_{s},R)$:

\begin{Thm}\label{GeneralTheorem}
Let $(p,q) \in \mathbb{Z}^2$. Assume that
\begin{itemize}
\item[(a)] $s \in R^{\times}$
\item[(b)] the Galois action of $\mu_s (k)$ on $H^{p,q}(Y_{s}\setminus F, R)$ is trivial
%\item[(c)] the structure morphism $X \rightarrow Spec(k)$ induces an isomorphism $H^{p,q}(X,R)\cong H^{p,q}(Spec(k),R)$
\end{itemize}
Then the morphism $\varphi_{s}: Y_{s} \rightarrow X$ induces an isomorphism $H^{p,q}(X,R)\cong H^{p,q}(Y_{s},R)$.
\end{Thm}

\begin{proof}
Consider the commutative diagram of localization sequences

\begin{center}
$\begin{xy}
  \xymatrix{
     H^{p-2,q-1} (F_{0},R) \ar[d] \ar[r] & H^{p,q} (X,R) \ar[d] \ar[r] & H^{p,q} (X \setminus F_{0},R) \ar[d] \ar[r] & H^{p-1,q-1} (F_{0},R) \ar[d]\\
     H^{p-2,q-1} (F,R) \ar[r] & H^{p,q} (Y_{s},R) \ar[r] & H^{p,q} (Y_{s} \setminus F,R) \ar[r] & H^{p-1,q-1} (F,R)}
\end{xy}$
\end{center}

The third vertical homomorphism from the left is a monomorphism by Corollary \ref{FixedSubgroup}; by Corollary \ref{Iso1} and assumption (b), it is also an epimorphism. Note that both $F$ and $F_{0}$ have pure codimension $1$ as subschemes of $Y_{s}$ and $X$ respectively by \cite[\href{https://stacks.math.columbia.edu/tag/0BCV}{Tag 0BCV}]{stacks-project}). Therefore, by \cite[Th{\'e}or{\`e}me 3.3]{D}, the vertical homomorphisms on the left-hand side and right-hand side are just $s \cdot {(\varphi_{s})}^{\ast}$, where $\varphi_{s}: F \xrightarrow{\cong} F_{0}$ is the isomorphism obtained from restricting the cyclic covering map $\varphi_{s}:Y_{s} \rightarrow X$; since $s \in R^{\times}$, it follows that the vertical homomorphisms on the left-hand side and right-hand side are isomorphisms. The four lemma now implies that the second vertical homomorphism from the left is an epimorphism; by Remark \ref{AlwaysInjective}, it is also a monomorphism. This finishes the proof.
\end{proof}

Now we can prove a theorem which gives a sufficient condition for the bijectivity for the homomorphism $\varphi_s^{\ast}: H^{p,q}(X,R) \rightarrow H^{p,q}(Y_{s},R)$:

\begin{Thm}\label{Ind-Thm}
Let $(p,q) \in \mathbb{Z}^2$. Assume that
\begin{itemize}
\item[(a)] $s \in R^{\times}$
\item[(b)] there is a $\mathbb{G}_{m,k}$-action $\gamma$ on $X$, which makes $f$ a quasi-invariant of weight $d \in \mathbb{Z}$ with respect to $\gamma$ and $\langle d,s \rangle = \mathbb{Z}$
\item[(c)] $H^{p-1,q-1}(Y_{s} \setminus F,R)=0$
\end{itemize}
Then the morphism $\varphi_{s}: Y_{s} \rightarrow X$ induces an isomorphism $H^{p,q}(X,R)\cong H^{p,q}(Y_{s},R)$.
\end{Thm}

\begin{proof}
Assumption (b) and Lemma \ref{Extension-Lem} guarantee that the Galois action of $\mu_{s}(k)$ is trivial on $H^{p,q}(Y_{s}\setminus F,R)$ as soon as the projection $\pi: (Y_{s}\setminus F)\times_{k}\mathbb{G}_{m,k} \rightarrow Y_{s} \setminus F$ induces an isomorphism
\begin{center}
$\pi^{\ast}: H^{p,q}(Y_{s}\setminus F, R) \xrightarrow{\cong} H^{p,q}((Y_{s}\setminus F)\times_{k}\mathbb{G}_{m,k}, R)$.
\end{center}
But $\pi$ has the factorization
\begin{center}
$\pi: (Y_{s}\setminus F) \times_k \mathbb{G}_{m,k} \hookrightarrow (Y_{s}\setminus F) \times_k \mathbb{A}^1_k \rightarrow Y_{s} \setminus F$,
\end{center}
where the first map is the inclusion and the second map is the projection on the first factor; the latter map induces an isomorphism on motivic cohomology by homotopy invariance. By Remark \ref{GM-Chow}, the first map always induces an injection on motivic cohomology and therefore it suffices to show that is also induces a surjection on motivic cohomology in bidegree $(p,q)$. This follows from assumption (c) and the exact localization sequence
\begin{center}
$...\rightarrow H^{p-2,q-1}(Y_{s}\setminus F,R)\rightarrow H^{p,q}((Y_{s}\setminus F)\times_{k}\mathbb{A}^{1}_{k},R)\rightarrow H^{p,q}((Y_{s}\setminus F)\times_{k}\mathbb{G}_{m,k},R) \rightarrow H^{p-1,q-1}(Y_{s}\setminus F,R)=0$
\end{center}
So the Galois action of $\mu_{s}(k)$ is trivial on $H^{p,q}(Y_{s}\setminus F,R)$. Now the statement follows from Theorem \ref{GeneralTheorem}.
\end{proof}

%\begin{Lem}\label{Chow-Lemma}
%With the notation of the discussion above, let $i \geq 1$ and assume that
%\begin{itemize}
%\item[(a)] $s \in R^{\times}$
%\item[(b)] $CH^i (X)\otimes_{\mathbb{Z}}R = 0$
%\end{itemize}
%Then $CH^i (Y_{s}) \otimes_{\mathbb{Z}}R =0$ if and only if $CH^i (Y_{s} \setminus F) \otimes_{\mathbb{Z}}R =0$.
%\end{Lem}

%\begin{proof}
%It follows from the localization sequence (tensored with $R$) that $CH^i (Y_{s} \setminus F) \otimes_{\mathbb{Z}}R =0$ as soon as $CH^i (Y_{s}) \otimes_{\mathbb{Z}}R =0$. Conversely, assume that $CH^i (Y_{s} \setminus F) \otimes_{\mathbb{Z}}R =0$. Then the Galois action on $CH^i (Y_{s} \setminus F) \otimes_{\mathbb{Z}}R =0$ is necessarily trivial, so Theorem \ref{GeneralTheorem} (which also works for Chow groups tensored with a coefficient ring $R$) implies that $CH^i (Y_{s}) \otimes_{\mathbb{Z}}R \cong CH^i (X) \otimes_{\mathbb{Z}}R =0$.
%As in the proof of Lemma \ref{Chow-Lemma}, it follows from a four lemma argument in the localization sequence and \cite[Th{\'e}or{\`e}me 3.3]{D} that $\varphi_{s}^{\ast}: H^{p,q}(X,R) \rightarrow H^{p,q}(Y_{s},R)$ is surjective; the injectivity follows from Remark \ref{graph}.
%\end{proof}

Theorem \ref{Ind-Thm} immediately implies the following theorem on the motivic cohomology groups in bidegree $(2i,i)$ for $i \in \mathbb{Z}$:

\begin{Thm}\label{Chow}
Assume that
\begin{itemize}
\item[(a)] $s \in R^{\times}$
\item[(b)] there is a $\mathbb{G}_{m,k}$-action $\gamma$ on $X$, which makes $f$ a quasi-invariant of weight $d \in \mathbb{Z}$ with respect to $\gamma$ and $\langle d,s \rangle = \mathbb{Z}$
%\item[(c)] $CH^i (X)\otimes_{\mathbb{Z}}R = 0$ for $i \geq 1$
\end{itemize}
Then $\varphi_{s}: Y_{s} \rightarrow X$ induces isomorphisms $H^{2i,i} (X,R) \cong H^{2i,i} (Y_{s},R)$ for $i \in \mathbb{Z}$.
\end{Thm}

\begin{proof}
Let $i \in \mathbb{Z}$ be arbitrary. Then the statement in the theorem follows from Theorem \ref{Ind-Thm} and the fact that $H^{2i-1,i-1} (Y_{s} \setminus F,R) = 0$ (cf. \cite[Vanishing Theorem 19.3]{MVW}).
%By assumption (c), the inclusions $X \setminus F_{0} \rightarrow X$ and $Y_{s} \setminus F \rightarrow Y_{s}$ induce isomorphisms
%\begin{center}
%$CH^{i} (X) \otimes_{\mathbb{Z}}R \xrightarrow{\cong} CH^{i} (X \setminus F_{0}) \otimes_{\mathbb{Z}}R$\\
%$CH^{i} (Y_{s}) \otimes_{\mathbb{Z}}R \xrightarrow{\cong} CH^{i} (Y_{s} \setminus F) \otimes_{\mathbb{Z}}R$.
%\end{center}
%By Lemma \ref{Chow-Lemma}, $CH^{i} (Y_{s}) \otimes_{\mathbb{Z}}R$ is trivial if and only if $CH^{i} (Y_{s} \setminus F) \otimes_{\mathbb{Z}}R$ is trivial. But by assumption (a) the latter group is trivial if and only if $CH^{i} (X \setminus F_{0}) \otimes_{\mathbb{Z}}R$ is trivial and the Galois action of $\mu_{s}(k)$ is trivial on $CH^{i} (Y_{s} \setminus F) \otimes_{\mathbb{Z}}R$. Triviality of $CH^{i} (X \setminus F_{0}) \otimes_{\mathbb{Z}}R$ follows from assumption (c).\\
%So it remains to be proven that the Galois action is trivial on $CH^{i} (Y_{s} \setminus F)\otimes_{\mathbb{Z}}R$. By assumption (b) and Remark \ref{Extension-Rem}, the triviality of the Galois action now follows from Lemma \ref{Extension-2} as soon as one verifies that the projection $\mathbb{G}_{m,k}\times_{k}(Y_{s}\setminus F) \rightarrow Y_{s}\setminus F$ induces an isomorphism
%\begin{center}
%$CH^i (Y_{s}\setminus F) \otimes_{\mathbb{Z}}R \xrightarrow{\cong} CH^i (\mathbb{G}_{m,k} \times_{k}(Y_{s}\setminus F)) \otimes_{\mathbb{Z}}R$.
%\end{center}
%But this follows from Remark \ref{GM-Chow}.
\end{proof}

The reasoning in the proofs of Theorems \ref{Ind-Thm} and \ref{Chow} also work for Chow groups tensored with a coefficient ring $R$ such that $s \in R^{\times}$:

\begin{Thm}\label{Chow-Tensor}
Assume that
\begin{itemize}
\item[(a)] $s \in R^{\times}$
\item[(b)] there is a $\mathbb{G}_{m,k}$-action $\gamma$ on $X$, which makes $f$ a quasi-invariant of weight $d \in \mathbb{Z}$ with respect to $\gamma$ and $\langle d,s \rangle = \mathbb{Z}$
%\item[(c)] $CH^i (X)\otimes_{\mathbb{Z}}R = 0$ for $i \geq 1$
\end{itemize}
Then $\varphi_{s}$ induces isomorphisms $CH^{i}(X)\otimes_{\mathbb{Z}}R \cong CH^{i}(Y_{s})\otimes_{\mathbb{Z}}R$ for $i \geq 0$.
\end{Thm}

\begin{proof}
One may use the general isomorphism $H^{2i,i}(V,R) \cong CH^i (V) \otimes_{\mathbb{Z}}R$ for $V \in Sm_k$ (cf. \cite[Preface]{MVW}) and then simply apply Theorem \ref{Chow}. Alternatively, one may simply use the isomorphism $H^{2i,i}(V,\mathbb{Z}) \cong CH^i (V)$ for $V \in Sm_k$ and realize that the reasoning in the proof of Theorem \ref{Ind-Thm} also works for Chow groups tensored with any coefficient ring $R$ such that $s \in R^{\times}$ as the sequence
\begin{center}
$...\rightarrow CH^{i-1}(Y_{s}\setminus F)\otimes_{\mathbb{Z}}R \rightarrow CH^{i}((Y_{s}\setminus F)\times_{k}\mathbb{A}^{1}_{k})\otimes_{\mathbb{Z}}R \rightarrow CH^{i}((Y_{s}\setminus F)\times_{k}\mathbb{G}_{m,k})\otimes_{\mathbb{Z}}R \rightarrow 0$
\end{center}
is always exact. %By considering $R=\mathbb{Z}/t\mathbb{Z}$ for $t \in \mathbb{Z}$, it follows in particular that $CH^{i}(Y_{s})$, $i \leq 1$, is a $t$-torsion group whenever $\langle s,t \rangle = \mathbb{Z}$ and $CH^{i}(X)=0$.
\end{proof}

\begin{Bsp}
A simple example satisfying the hypotheses of Theorem \ref{Chow-Tensor} is $X = \mathbb{A}^n_{k} = Spec (k[x_{1},...,x_{n}])$ with $\mathbb{G}_{m,k}$-action given by $(\lambda, x_{1},...,x_{n}) \mapsto (x_{1} \lambda^{d_{1}},...,x_{n} \lambda^{d_{n}})$ for some $d_{j} \in \mathbb{Z}$. In this case any coordinate function $x_{j}$ is a quasi-invariant of weight $d_{j}$ and, if we fix some $j \in \{1,...,n\}$, the cyclic covering $Y_{s}$ of $X$ of order $s$ with respect to the coordinate function $x_{j}$ is again isomorphic to an affine space of dimension $n$ over $k$. If $s \in R^{\times}$ and $\langle d_{j}, s \rangle = \mathbb{Z}$, then Theorem \ref{Chow-Tensor} implies that $CH^{i}(X)\otimes_{\mathbb{Z}}R \cong CH^{i}(Y_{s})\otimes_{\mathbb{Z}}R$ for $i \geq 0$ (which is clear as both $X$ and $Y_{s}$ are affine spaces of dimension $n$ over $k$).
\end{Bsp}

Non-trivial examples satisfying the hypotheses of Theorem \ref{Chow-Tensor} are given in Section \ref{Examples}.

\begin{Rem}\label{Chow-Torsion}
Under the assumptions of Theorem \ref{Chow}, the reasoning in the proof of Theorem \ref{Ind-Thm} actually shows that the Galois action is also trivial on $CH^{i}(Y_{s}\setminus F)$ for $i \geq 1$ (even without tensoring with a coefficient ring $R$ such that $s \in R^{\times}$) and hence the composite $\mu \circ \varphi_s$ from Corollary \ref{Existence} or Remark \ref{graph} induces multiplication with $s$ on the Chow groups $CH^{i}(Y_{s}\setminus F)$ for $i \geq 1$. In conclusion, if $CH^{i}(X)=0$, $CH^{i}(Y_{s})\cong CH^{i}(Y_{s}\setminus F)$ and the assumptions of Theorem \ref{Chow} are satisfied, then it follows that $CH^{i}(Y_{s})$ is an $s$-torsion group.
\end{Rem}

We can now prove several corollaries to Theorems \ref{Ind-Thm}, \ref{Chow} and Theorem \ref{Chow-Tensor}:

\begin{Kor}\label{Euler-2}
Under the assumptions of Theorem \ref{Chow}, further assume that $d := dim (Y_{s}) \geq 3$ and that $CH^d (X)\otimes_{\mathbb{Z}} \mathbb{Q} = 0$. Then every vector bundle $Y_s$ of rank $d$ has a trivial direct summand of rank $1$.
\end{Kor}

\begin{proof}
The statement follows from Theorem \ref{Chow-Tensor} with $R=\mathbb{Q}$ and from Proposition \ref{Euler}.
\end{proof}

\begin{Kor}\label{rank-hom}
Under the assumptions of Theorem \ref{Chow}, the morphism $\varphi_s$ induces an isomorphism $K_0 (X) \otimes \mathbb{Q} \cong K_0 (Y_{s}) \otimes \mathbb{Q}$.
\end{Kor}

\begin{proof}
The statement follows from Theorem \ref{Chow-Tensor} and the isomorphism
\begin{center}
$\bigoplus_{i=0}^{d} CH^i (U) \otimes \mathbb{Q} \cong K_0 (U) \otimes \mathbb{Q}$
\end{center}
for any smooth affine scheme of dimension $d$ over $k$ (e.g. see \cite[Example 15.2.16]{Fu}).
\end{proof}

\begin{Kor}\label{H53}
Assume that
\begin{itemize}
\item[(a)] $s \in R^{\times}$
\item[(b)] there is a $\mathbb{G}_{m,k}$-action $\gamma$ on $X$, which makes $f$ a quasi-invariant of weight $d \in \mathbb{Z}$ with respect to $\gamma$ and $\langle d,s \rangle = \mathbb{Z}$
\item[(c)] $H^{2i,i} (X,R) = 0$ for $i \geq 1$
\end{itemize}
Then the morphism $\varphi_{s}: Y_{s} \rightarrow X$ induces an isomorphism $H^{2i+1,i+1}(X,R) \cong H^{2i+1,i+1}(Y_{s},R)$ for $i \geq 1$.
\end{Kor}

\begin{proof}
Follows from Theorem \ref{Chow} and Theorem \ref{Ind-Thm}.
\end{proof}

\begin{Kor}\label{FiniteCoefficients}
Let $R = \mathbb{Z}/n\mathbb{Z}$. Assume that
\begin{itemize}
\item[(a)] $\langle s,n \rangle = \mathbb{Z}$
\item[(b)] there is a $\mathbb{G}_{m,k}$-action $\gamma$ on $X$, which makes $f$ a quasi-invariant of weight $d \in \mathbb{Z}$ with respect to $\gamma$ and $\langle d,s \rangle = \mathbb{Z}$
\item[(c)] $H^{p,q}(X,R) \cong H^{p,q}(Spec(k),R)$ for all $(p,q) \in \mathbb{Z}^2$
\item[(d)] $H^{p,q}(F,R) \cong H^{p,q}(Spec(k),R)$ for all $(p,q) \in \mathbb{Z}^2$
\end{itemize}
Then $H^{p,q}(Y_{s},R)=H^{p,q}(Spec(k),R)=0$ for all $(p,q) \in \mathbb{Z}^2$ such that $p > q \geq 1$ or $p < 0$.
\end{Kor}

\begin{proof}
One can assume $p \leq 2q$. Then the statement follows inductively from Theorem \ref{Chow}, Theorem \ref{Ind-Thm} and the fact that $H^{p,q}(Spec(k),R)=0$ for $p\neq 0$ (cf. \cite[Preface]{MVW}): Note that, for each $(p,q) \in \mathbb{Z}^2$ with $p \leq 2q$, the base case is bidegree $(2p-2q,p-q) \in \mathbb{Z}^2$ (which is covered by Theorem \ref{Chow}) and applying Theorem \ref{Ind-Thm} exactly $2q-p$ times yields the claim for bidegree $(p,q)$.
\end{proof}

\begin{Rem}\label{FiniteCoefficients-2}
By inspection of the proof of Corollary \ref{FiniteCoefficients}, it follows that one may replace assumptions (c) and (d) by
\begin{itemize}
\item[(c')] $H^{p,q}(X,R) \cong H^{p,q}(Spec(k),R)$ for all $(p,q) \in \mathbb{Z}^2$ such that $p > q \geq 1$
\item[(d')] $H^{p,q}(F,R) \cong H^{p,q}(Spec(k),R)$ for all $(p,q) \in \mathbb{Z}^2$ such that $p > q \geq 1$
\end{itemize}
in order to deduce that $H^{p,q}(Y_{s},R)=H^{p,q}(Spec(k),R)=0$ for all $(p,q) \in \mathbb{Z}^2$ such that $p > q \geq 1$.% (Note that $H^{p,0} (F,R) = 0$ for $p \neq 0$ by \cite[Preface]{MVW}.)
\end{Rem}

Recall that the cohomology group $H_{Nis}^{2}(Y,\textbf{I}^3)$ appears in the cohomological criterion in Theorem \ref{Syed} for the triviality of all vector bundles over any smooth affine fourfold $Y$ over $k$. In this context, Corollary \ref{FiniteCoefficients} enables us to prove the following result:

\begin{Kor}\label{FundamentalIdeal}
Under the assumptions of Corollary \ref{H53} for $R = \mathbb{Z}/2\mathbb{Z}$, further assume that $\dim (Y_{s}) = 4$. Then $H^2_{Nis}(X,\textbf{I}^3) \cong H_{Nis}^{2}(Y_{s},\textbf{I}^3)$.
\end{Kor}

\begin{proof}
Let $\overline{\textbf{I}}^j = \textbf{I}^{j}/\textbf{I}^{j+1}$ for any integer $j \geq 0$. It follows from \cite[Theorem 1.3]{Tot} and the affirmation of the Milnor conjectures on the mod $2$ norm residue homomorphism and quadratic forms that there are long exact sequences of the form
\begin{center}
$... \rightarrow H^{2+j,j-1}(V,\mathbb{Z}/2\mathbb{Z}) \rightarrow H^{2+j,j}(V,\mathbb{Z}/2\mathbb{Z}) \rightarrow H^{2}_{Nis}(V,\overline{\textbf{I}}^j) \rightarrow H^{3+j,j-1}(V,\mathbb{Z}/2\mathbb{Z}) \rightarrow ...$
\end{center}
for any smooth affine fourfold $V$ over $k$. Now \cite[Vanishing Theorem 19.3]{MVW} directly implies that the long exact sequence above for $j=3$ yields an isomorphism $H^{5,3}(V,\mathbb{Z}/2\mathbb{Z}) \cong H^{2}_{Nis}(V,\overline{\textbf{I}}^3)$ for any smooth affine fourfold $V$ over $k$.\\
Note that $\textbf{I}^4 = \overline{\textbf{I}}^4$ over $V$ by \cite[Proposition 5.1]{AF1}, so the proof of \cite[Lemma 4.0.3]{F} shows that $H^{2}_{Nis}(V, \textbf{I}^4) \cong H^{2}_{Nis}(V, \overline{\textbf{I}}^4) \cong H^{2}_{Nis}(V, \textbf{K}_{4}^{M}/2) = 0$ as well as $H^{3}_{Nis}(V, \textbf{I}^4) \cong H^{3}_{Nis}(V, \overline{\textbf{I}}^4) \cong H^{3}_{Nis}(V, \textbf{K}_{4}^{M}/2) = 0$. Now the long exact sequence of cohomology groups associated to the short exact sequence of sheaves
\begin{center}
$0 \rightarrow \textbf{I}^4 \rightarrow \textbf{I}^3 \rightarrow \overline{\textbf{I}}^3 \rightarrow 0$
\end{center}
over the small Nisnevich site of $V$ shows that $H^{2}_{Nis}(V, \textbf{I}^3) \cong H^{2}_{Nis}(V, \overline{\textbf{I}}^3)$ and, in particular, $H^{5,3}(V,\mathbb{Z}/2\mathbb{Z}) \cong H^{2}_{Nis}(V,\textbf{I}^3)$ by the first paragraph of this proof.\\
Finally, as both $X$ and $Y_{s}$ are smooth affine fourfolds over $k$, Corollary \ref{H53} shows that $\varphi_{s}$ induces an isomorphism $H^{5,3}(X,\mathbb{Z}/2\mathbb{Z}) \cong H^{5,3}(Y_{s},\mathbb{Z}/2\mathbb{Z})$ and therefore we obtain an isomorphism $H^2_{Nis}(X,\textbf{I}^3) \cong H_{Nis}^{2}(Y_{s},\textbf{I}^3)$. This finishes the proof.
\end{proof}

%\begin{Rem}\label{FundamentalIdeal-2}
%By inspection of the proof of Corollary \ref{FundamentalIdeal}, one sees easily that it is sufficient the replace the assumptions (c) and (d) from Theorem \ref{FiniteCoefficients} by assumptions (c') and (d') of Remark \ref{FiniteCoefficients-2}.
%\end{Rem}

Note that the group $\mu_s (k)$ also acts on $Y_s$ and not only on $Y_s \setminus F$. By abuse of language, we also call the corresponding action the Galois action on $Y_s$.

\begin{Prop}\label{Equiv}
Let $(p,q) \in \mathbb{Z}^2$ and assume $s \in R^{\times}$.  Then the Galois action of $\mu_s (k)$ is trivial on $H^{p,q}(Y_{s}, R)$ if it is trivial on $H^{p,q}(Y_{s}\setminus F, R)$. Furthermore, the converse statement holds if the inclusion $Y_{s} \setminus F \rightarrow Y_s$ induces an isomorphism on motivic cohomology with coefficients in $R$ in bidegree $(p,q)$.
\end{Prop}

\begin{proof}
The first statement follows from Theorem \ref{GeneralTheorem} and the commutative diagram

\begin{center}
$\begin{xy}
\xymatrix{
H^{p,q}(X,R) \ar[d]_{{id}_{X}^{\ast}} \ar[r]^{\varphi_{s}^{\ast}}_{\cong} & H^{p,q}(Y_{s},R) \ar[d]^{g^{\ast}}_{\cong}\\
H^{p,q}(X,R) \ar[r]_{\varphi_{s}^{\ast}}^{\cong} & H^{p,q}(Y_s,R)
}
\end{xy}$
\end{center}
for all $g \in \mu_s (k)$. The second statement follows directly from the commutative diagram

\begin{center}
$\begin{xy}
\xymatrix{
H^{p,q}(Y_{s},R) \ar[d]_{i^{\ast}}^{\cong} \ar[r]^{g^{\ast}} & H^{p,q}(Y_{s},R) \ar[d]^{i^{\ast}}_{\cong}\\
H^{p,q}(Y_s \setminus F,R) \ar[r]_{g^{\ast}} & H^{p,q}(Y_s \setminus F,R)
}
\end{xy}$
\end{center}
for all $g \in \mu_s (k)$, where $i: Y_{s} \setminus F \rightarrow Y_{s}$ denotes the inclusion.
\end{proof}

\begin{Rem}
If only the variety $X$ is assumed to be smooth, the variety $Y_s \setminus F$ is smooth even without any assumptions on $F$; then the analogue of Theorem \ref{Chow} (and Theorem \ref{Chow-Tensor}) hold for $Y_s \setminus F$: Assume that
\begin{itemize}
\item[(a)] $s \in R^{\times}$
\item[(b)] there is a $\mathbb{G}_{m,k}$-action $\gamma$ on $X$, which makes $f$ a quasi-invariant of weight $d \in \mathbb{Z}$ with respect to $\gamma$ and $\langle d,s \rangle = \mathbb{Z}$
%\item[(c)] $F_{0}$ is smooth over $k$ and $CH^i (F_{0})\otimes_{\mathbb{Z}}R =0$ for $i \geq 1$
\end{itemize}
Then $\varphi_{s}$ induces an isomorphism $H^{2i,i} (X \setminus F_{0},R) \cong H^{2i,i} (Y_{s} \setminus F,R)$ for $i \geq 1$.
\end{Rem}

Let $s,t \geq 1$ be integers. With the notation of Definition \ref{CyclicCoverings} and under the General Assumptions \ref{Assumptions}, we call the subvariety $Y_{s,t}$ of $X \times_k \mathbb{A}^1_k \times_k \mathbb{A}^1_k$ defined by the equation $\{f + u^{s} + v^{t}=0\}$ a bicyclic covering of orders $s,t$ (here $u$ and $v$ are the variables of the two copies of $\mathbb{A}^1_k$). We conclude the section with some results on the Chow groups and on the vector bundles of such bicyclic coverings. At first, we prove a result on the Chow groups of such bicyclic coverings:

\begin{Thm}\label{Bicyclic}
With the notation of Definition \ref{CyclicCoverings} and under the General Assumptions \ref{Assumptions}, further assume that
\begin{itemize}
\item[(a)] $f \in O_X (X)$ is prime and $O_X (X)$ is a UFD
\item[(b)] the smooth variety $X$ has the Chow groups of $Spec(k)$
\item[(c)] the smooth variety $F_{0}=\{f=0\} \subset X$ has the Chow groups of $Spec(k)$
\item[(d)] there is a $\mathbb{G}_{m,k}$-action $\gamma$ on $X$, which makes $f$ a quasi-invariant of weight $d \in \mathbb{Z}$ with respect to $\gamma$
\item[(e)] $\langle d,s \rangle = \langle d,t \rangle = \langle s,t \rangle= \mathbb{Z}$
%\item[(f)] $CH^1 (Y_{s}) \cong CH^1 (Y_{t}) = 0$
\end{itemize}
Then $CH^i (Y_{s,t})=0$ for $i \geq 1$, where $Y_{s,t}:=\{f + u^{s} + v^{t}=0\} \subset X \times_k \mathbb{A}^1_k \times_k \mathbb{A}^1_k$.
\end{Thm}

\begin{proof}
The variety $Y_{s,t}$ is a cyclic covering of the product of $X$ and $\mathbb{A}^1_{k}$ (with coordinate $u$) along $Y_{s}=\{-f-u^{s}=0\}$ of order $t$; note that $Y_s$ is a cyclic covering of $X$ branched to order $s$ as usual, but here we use the regular function $-f$ instead of $f$. Theorem \ref{Chow-Tensor} together with assumptions (d), (e) and Remark \ref{Extension-Bicyclic} implies that $CH^{i}(Y_{s,t})$, $i \geq 1$, are divisible by any integer $n$ with $\langle n,t\rangle=\mathbb{Z}$ (just take $R = \mathbb{Z}/n\mathbb{Z}$). Now realize that $Y_{s,t}$ is also a cyclic covering of the product of $X$ and $\mathbb{A}^1_{k}$ (with coordinate $v$) along $Y_{t}=\{-f-v^{t}=0\}$ of order $s$. Using Theorem \ref{Chow-Tensor} again yields that $CH^{i}(Y_{s,t})$, $i \geq 1$, are also divisible by any integer $n$ with $\langle n,s\rangle=\mathbb{Z}$ (again, just take $R = \mathbb{Z}/n\mathbb{Z}$). As $s$ and $t$ are coprime, the groups $CH^{i}(Y_{s,t})$, $i \geq 1$, are divisible by any prime number and hence are divisible groups. In conclusion, it suffices to show that there is an integer $n \neq 0$ such that $CH^{i}(Y_{s,t})$ is $n$-torsion.\\
For this purpose, note first that Remark \ref{Chow-Torsion} together with assumption (b) and Remark \ref{Extension-Bicyclic} implies that $CH^{i}(Y_{s,t} \setminus Y_{s})$ is a $t$-torsion group for $i \geq 1$. Furthermore, the localization sequence
\begin{center}
$CH^{0}(Y_{s}) \rightarrow CH^{1} (Y_{s,t}) \rightarrow CH^{1} (Y_{s,t}\setminus Y_{s}) \rightarrow 0$
\end{center}
implies that $CH^1 (Y_{s,t})$ is $mt$-torsion for some integer $m \neq 0$: The group $CH^0 (Y_{s}) \cong \mathbb{Z}$ cannot map injectively into $CH^{1}(Y_{s,t})$ because tensoring the localization sequence with $\mathbb{Q}$ would then yield $CH^{1}(Y_{s,t}) \otimes_{\mathbb{Z}} \mathbb{Q} \neq 0$, which is impossible because of Theorem \ref{Chow-Tensor} with $R= \mathbb{Q}$. Therefore one obtains a short exact sequence of the form
\begin{center}
$0 \rightarrow \mathbb{Z}/m\mathbb{Z} \rightarrow CH^{1} (Y_{s,t}) \rightarrow CH^{1} (Y_{s,t}\setminus Y_{s}) \rightarrow 0$
\end{center}
for some $m \neq 0$, which shows that $CH^1 (Y_{s,t})$ is indeed $mt$-torsion. This settles the case $i=1$.\\
Furthermore, for $i \geq 2$, Remark \ref{Chow-Torsion} together with assumptions (b) and (c) implies that $CH^{i-1}(Y_{s})$ is an $ms$-torsion group for some integer $m \neq 0$ (which might depend on $i$): Indeed, Remark \ref{Chow-Torsion} implies that $CH^{i-1}(Y_{s}\setminus F)$ is $s$-torsion and the exact localization sequence
\begin{center}
$CH^{i-2}(F) \rightarrow CH^{i-1} (Y_{s}) \rightarrow CH^{i-1} (Y_{s}\setminus F) \rightarrow 0$
\end{center}
shows that $CH^{i-1}(Y_{s})$ is also $s$-torsion if $i \geq 3$ because of assumption (c) and at least $ms$-torsion for some integer $m \neq 0$ if $i=2$; in the latter case $CH^0 (F) \cong \mathbb{Z}$ cannot map injectively into $CH^{i-1}(Y_{s})$ because tensoring the localization sequence with $\mathbb{Q}$ would then yield $CH^{1}(Y_{s}) \otimes_{\mathbb{Z}} \mathbb{Q} \neq 0$, which is impossible because of Theorem \ref{Chow-Tensor}. Then the short exact sequence of the form
\begin{center}
$0 \rightarrow \mathbb{Z}/m\mathbb{Z} \rightarrow CH^{i-1} (Y_{s}) \rightarrow CH^{i-1} (Y_{s}\setminus F) \rightarrow 0$
\end{center}
for some $m \neq 0$ shows that $CH^{i-1} (Y_{s})$ is indeed $mt$-torsion for $i=2$ as well.\\
Altogether, the previous paragraphs and the exact localization sequence
\begin{center}
$CH^{i-1}(Y_{s}) \rightarrow CH^{i} (Y_{s,t}) \rightarrow CH^{i} (Y_{s,t}\setminus Y_{s}) \rightarrow 0$
\end{center}
imply that $CH^{i}(Y_{s,t})$ is $n$-torsion for $n=mst$ whenever $i \geq 2$ (note again that $m \neq 0$ might depend on $i$). This concludes the proof.
\end{proof}

\begin{Rem}\label{Bicyclic-Contractibility}
If one further assumes that $k=\mathbb{C}$ and that $X$ is topologically contractible in Theorem \ref{Bicyclic}, then $Y_{s,t}$ is topologically contractible by \cite[Example 6.2]{Z}. %\cite[\href{https://stacks.math.columbia.edu/tag/020J}{Tag 020J}]{stacks-project}
\end{Rem}

\begin{Bsp}\label{Bicyclic-Example-Affine-Space}
Again, a simple example is $X = \mathbb{A}^n_{k} = Spec (k[x_{1},...,x_{n}])$ with $\mathbb{G}_{m,k}$-action given by $(\lambda, x_{1},...,x_{n}) \mapsto (x_{1} \lambda^{d_{1}},...,x_{n} \lambda^{d_{n}})$ for some $d_{j} \in \mathbb{Z}$. As explained before, any coordinate function $x_{j}$ is then a quasi-invariant of weight $d_{j}$. If we then fix some $j \in \{1,...,n\}$, the bicyclic covering $Y_{s,t}$ of $X$ of orders $s,t$ with respect to the coordinate function $x_{j}$ is isomorphic to an affine space of dimension $n+1$ over $k$. If $s, t$ and $d_{j}$ are pairwise coprime, then Theorem \ref{Bicyclic} recovers the fact that $CH^{i}(Y_{s,t}) = 0$ for $i \geq 1$.
\end{Bsp}

A non-trivial example satisfying the hypotheses of Theorem \ref{Bicyclic} is given in Section \ref{BicyclicFourfold}. Finally, we can give a sufficient criterion for the triviality of all vector bundles over a bicyclic covering as in Theorem \ref{Bicyclic} if the dimension of the bicyclic covering is $4$:

\begin{Thm}\label{Bicyclic-VB}
With the notation of Definition \ref{CyclicCoverings} and the General Assumptions \ref{Assumptions}, further assume that
\begin{itemize}
\item[(a)] $f \in O_X (X)$ is prime and $O_X (X)$ is a UFD
\item[(b)] the smooth variety $X$ has the Chow groups of $Spec(k)$
\item[(c)] the smooth variety $F_{0}=\{f=0\} \subset X$ has the Chow groups of $Spec(k)$
\item[(d)] there is a $\mathbb{G}_{m,k}$-action $\gamma$ on $X$, which makes $f$ a quasi-invariant of weight $d \in \mathbb{Z}$ with respect to $\gamma$
\item[(e)] $\langle d,s \rangle = \langle d,t \rangle = \langle s,t \rangle= \mathbb{Z}$
\item[(f)] $\dim (Y_{s,t})=4$
%\item[(f)] $CH^1 (Y_{s}) \cong CH^1 (Y_{t}) = 0$
\end{itemize}
Then $Y_{s,t}:=\{f + u^{s} + v^{t}=0\} \subset X \times_k \mathbb{A}^1_k \times_k \mathbb{A}^1_k$ has only trivial vector bundles and $W_{SL}(Y_{s,t})=0$.
\end{Thm}

\begin{Rem}
Recall that the group $W_{SL} (R)$ of a commutative ring $R$ was defined in \cite[\S 3]{SV} and therefore we can define $W_{SL}(X) = W_{SL}(\mathcal{O}_{X}(X))$ for affine schemes. We also refer the reader to \cite[Section 2.B]{Sy} for a brief introduction to this group.
\end{Rem}

\begin{proof}
Since $\dim (Y_{s,t})=4$, the triviality of all stably trivial vector bundles of rank $2$ over $Y_{s,t}$ is equivalent to the triviality of $W_{SL}(Y_{s,t})$ (cf. \cite[Theorem 3.19]{Sy}). So it suffices to prove the first statement in the theorem. The variety $Y_{s,t}$ has trivial Chow groups by Theorem \ref{Bicyclic}. In particular, $Y_{s,t}$ has only stably trivial vector bundles. Since $\dim (Y_{s,t})=4$, it now suffices by Theorem \ref{Syed} to show that $H^{2}_{Nis}(Y_{s,t},\textbf{I}^3)$ is trivial.\\
Note that either $s$ or $t$ has to be odd, so $Y_{s,t}$ is a cyclic covering of $X \times_k \mathbb{A}^1_k$ of odd order in any case. As $H^{2i,i}(X, \mathbb{Z}/2\mathbb{Z}) \cong CH^{i}(X) \otimes_{\mathbb{Z}} \mathbb{Z}/2\mathbb{Z} = 0$ for $i \geq 1$ by assumption, Corollary \ref{FundamentalIdeal} therefore implies that $H_{Nis}^{2}(Y_{s,t}, \textbf{I}^3) \cong H_{Nis}^{2}(X \times_k \mathbb{A}^1_k, \textbf{I}^3)$. But $H_{Nis}^{2}(X \times_k \mathbb{A}^1_k, \textbf{I}^3) \cong H_{Nis}^{2}(X, \textbf{I}^3) = 0$, where the first isomorphism follows from the fact that $\textbf{I}^3$ is a strictly $\mathbb{A}^1$-invariant sheaf of abelian groups (cf. \cite[Section 2]{AF1}) and the triviality of $H_{Nis}^{2}(X, \textbf{I}^3)$ follows from \cite[Proposition 5.2]{AF1} as $X$ is a smooth affine threefold over $k$. This finishes the proof. 
%We can first apply Corollary \ref{FiniteCoefficients} in order to deduce that $H^{p,q}(Y_{s},\mathbb{Z}/2\mathbb{Z}) = 0$ for any $(p,q) \in \mathbb{Z}^2$ with $p > q \geq 1$. Then we can apply Corollary \ref{FundamentalIdeal} (and Remark \ref{FundamentalIdeal-2}) to obtain that indeed $H^{2}_{Nis}(Y_{s,t},\textbf{I}^3) = 0$. This finishes the proof.
\end{proof}

\begin{Bsp}
If we let $n=3$ in Example \ref{Bicyclic-Example-Affine-Space}, then Theorem \ref{Bicyclic-VB} recovers the fact that all vector bundles over $\mathbb{A}^4_k$ are trivial.
\end{Bsp}

A non-trivial example satisfying the hypotheses of Theorem \ref{Bicyclic-VB} is given in Section \ref{BicyclicFourfold}.

\section{Examples}\label{Examples}

Finally, we are able to discuss some examples of smooth affine varieties which are cyclic coverings of varieties considered in Example \ref{KREX} and of varieties considered in \cite{DPO}. As usual, we assume that the base field $k$ is algebraically closed of characteristic $0$.

\subsection{Cyclic coverings of Koras-Russell threefolds of the first kind}\label{EXKR1}

%\begin{Bsp}\label{EXKR1}
Consider the Koras-Russell threefold $Y_s$ defined by the equation

\begin{center}
$x + x^2 y^{\alpha_1} + z^{\alpha_2} + t^{\alpha_3}$
\end{center}

with $\alpha_{1}, \alpha_{2}, \alpha_{3} \geq 2$ pairwise coprime (cf. \cite[Proposition 6.2]{Z}). This is a cyclic covering of the Koras-Russell threefold $X$ of the first kind defined by the equation

\begin{center}
$x + x^2 y + z^{\alpha_2} + t^{\alpha_3}$
\end{center}

with $f = y \in \mathcal{O}_X (X)$ and $s=\alpha_1$ (see Example \ref{KREX}). There is a $\mathbb{G}_{m,k}$-action on $X$ given by $(\lambda,x,y,z,t) \mapsto (\lambda^{\alpha_{2}\alpha_{3}}x,\lambda^{-\alpha_{2}\alpha_{3}}y,\lambda^{\alpha_{3}}z,\lambda^{\alpha_{2}}t)$ which makes $f=y$ a quasi-invariant of weight $-\alpha_{2}\alpha_{3}$.\\
As $X$ is $\mathbb{A}^1$-contractible (cf. \cite{DF}), the structure morphism $X \rightarrow Spec(k)$ induces an isomorphism on motivic cohomology with coefficients in any commutative ring $R$; in particular, the Chow groups of $X$ are that of $Spec(k)$ and the coordinate ring $\mathcal{O}_{X}(X)$ is a UFD. The subscheme $F_0 = \{y=0\} \subset X$ is isomorphic to $\mathbb{A}^2_k$ and therefore $f=y$ is prime in $\mathcal{O}_{X}(X)$; moreover, it follows that $F_0 \rightarrow Spec(k)$ induces an isomorphism on motivic cohomology with coefficients in $R$. The analogous statement holds for Chow groups of $X$ and $F_{0}$ tensored with $R$; in particular, for $i \geq 1$, it follows that $CH^i (F_{0})\otimes_{\mathbb{Z}}R$ and $CH^i (X)\otimes_{\mathbb{Z}}R$ are trivial. If $s=\alpha_{1} \in R^{\times}$, Theorem \ref{Chow-Tensor} implies that $CH^i (Y_{s})\otimes_{\mathbb{Z}}R =0$ for $i \geq 1$.\\
Algebraic vector bundles over $Y_s$ are completely classified by their Chern classes; in fact, all vector bundles over $Y_{s}$ are trivial if and only if $CH^1 (Y_s)$, $CH^2 (Y_s)$ and $CH^3 (Y_s)$ are trivial (cf. \cite{AF1} and \cite{KM}). So it follows from Theorem \ref{Chow-Tensor} with $R=\mathbb{Q}$ that "rationally" all vector bundles over $Y_s$ are trivial, i.e., $CH^1 (Y_{s})\otimes_{\mathbb{Z}}\mathbb{Q}$, $CH^2 (Y_{s})\otimes_{\mathbb{Z}}\mathbb{Q}$ and $CH^3 (Y_{s})\otimes_{\mathbb{Z}}\mathbb{Q}$ are trivial; similarly, it follows from Corollary \ref{rank-hom}, that $K_{0} (Y_{s}) \otimes_{\mathbb{Z}}\mathbb{Q} \cong \mathbb{Q}$, where the isomorphism is induced by the rank homomorphism. Corollary \ref{Euler-2} implies that any rank $3$ bundle over $Y_s$ admits a free direct summand of rank $1$.\\
If $k = \mathbb{C}$, then $Y_{s}$ is topologically contractible (it is a Koras-Russell threefold by \cite[Proposition 6.2]{Z}); it is well-known that topologically contractible smooth affine complex varieties have a trivial Picard group and hence $CH^1 (Y_{s}) \cong CH^1 (Y_{s}) \otimes_{\mathbb{Z}}R = 0$ (cf. \cite[Theorem 1]{G}). If $\alpha_{1},\alpha_{2},\alpha_{3} >> 1$, then \cite[Proposition 6.2]{Z} shows that the logarithmic Kodaira dimension of $Y_s$ is $2$ and therefore the variety $Y_s$ has to be a Koras-Russell threefold of the third kind in this case as Koras-Russell threefolds of the first or second kind are known to have logarithmic Kodaira dimension $-\infty$ (see Example \ref{KREX}).
%\end{Bsp}

\subsection{Cyclic coverings of Koras-Russell threefolds of the second kind}

%\begin{Bsp}
Consider variety $Y_s$ defined by the equation

\begin{center}
$x + {(x^b + z^{\alpha_2})}^{l} y^{\alpha_1} + t^{\alpha_3}$
\end{center}

with $\alpha_{1}, \alpha_{2}, \alpha_{3} \geq 2$ pairwise coprime, $b, l \geq 2$ with $\langle b, \alpha_{2} \rangle = \mathbb{Z}$ and $\langle bl-1, \alpha_{1} \rangle = \mathbb{Z}$. The variety $Y_s$ is a cyclic covering of the Koras-Russell threefold $X$ of the second kind defined by the equation

\begin{center}
$x + {(x^b + z^{\alpha_2})}^{l} y + t^{\alpha_3}$
\end{center}

with $f = y \in \mathcal{O}_X (X)$ and $s=\alpha_1$ (see Example \ref{KREX}). There is a $\mathbb{G}_{m,k}$-action on $X$ given by $(\lambda,x,y,z,t) \mapsto (\lambda^{\alpha_{2}\alpha_{3}}x,\lambda^{(1-bl)\alpha_{2}\alpha_{3}}y,\lambda^{b\alpha_{3}}z,\lambda^{\alpha_{2}}t)$ which makes $f=y$ a quasi-invariant of weight $(1-bl)\alpha_{2}\alpha_{3}$. As pointed out in Example \ref{KREX}, $X$ is stably $\mathbb{A}^1$-contractible; in particular, the structure morphism $X \rightarrow Spec(k)$ induces an isomorphism on Chow groups and therefore the coordinate ring $\mathcal{O}_{X}(X)$ is a UFD. The subscheme $F_0 = \{y=0\} \subset X$ is isomorphic to $\mathbb{A}^2_k$ and hence $f=y$ is prime in $\mathcal{O}_{X}(X)$ and $F_0 \rightarrow Spec(k)$ induces an isomorphism on Chow groups; in particular, for $i \geq 1$, it follows that $CH^i (F_{0})\otimes_{\mathbb{Z}}R$ and $CH^i (X)\otimes_{\mathbb{Z}}R$ are trivial for any coefficient ring $R$. If $s=\alpha_{1} \in R^{\times}$, then Theorem \ref{Chow-Tensor} implies that in fact $CH^i (Y_{s})\otimes_{\mathbb{Z}}R =0$ for $i \geq 1$.\\
As in the previous example, it follows again from Theorem \ref{Chow-Tensor} that "rationally" all vector bundles over $Y_s$ are trivial as $CH^1 (Y_{s})\otimes_{\mathbb{Z}}\mathbb{Q}$, $CH^2 (Y_{s})\otimes_{\mathbb{Z}}\mathbb{Q}$ and $CH^3 (Y_{s})\otimes_{\mathbb{Z}}\mathbb{Q}$ are trivial groups; moreover, Corollary \ref{rank-hom} implies that $K_{0} (Y_{s}) \otimes_{\mathbb{Z}}\mathbb{Q} \cong \mathbb{Q}$ and Corollary \ref{Euler-2} implies that any rank $3$ bundle over $Y_s$ admits a free direct summand of rank $1$.\\
If $k = \mathbb{C}$, it follows directly from the construction in \cite[Theorem 4.1]{KR} that the variety $Y_s$ is a Koras-Russell threefold; indeed, choose $-a'_1 = dl-1$, $a'_2 = d$, $a'_3 = 1$, $\bar{u}_2 = \xi$, $\bar{u}_3 = \zeta + {(\zeta^d + \xi)}^{l}$ in \cite[Theorem 4.1]{KR} in order to obtain the type of equation defining $Y_s$. In particular, $Y_{s}$ is topologically contractible and has a trivial Picard group, i.e., $CH^1 (Y_{s}) \cong CH^1 (Y_{s}) \otimes_{\mathbb{Z}}R = 0$. If $\alpha_{2}, \alpha_{3} >> 1$ and $\alpha_{1} \geq (dl-1)\alpha_{2} \alpha_{3}$, this Koras-Russell threefold has logarithmic Kodaira dimension $2$ by \cite[Proposition 6.5]{KR} and hence must be a Koras-Russell threefold of the third kind.
%\end{Bsp}
%Cf. the end of the proof of Lemma 3.4 in "On the Russell-Koras contractible threefolds" by Kaliman and Makar-Limanov

\subsection{Cyclic coverings of stably $\mathbb{A}^1$-contractible smooth affine fourfolds}

%\begin{Bsp}
Consider the variety $Y_s$ defined by the equation

\begin{center}
$x + x^n z^{\alpha_{0}} + y_1^{\alpha_1} + y_2^{\alpha_2} + y_3^{\alpha_3}$
\end{center}

with $\alpha_{0}, \alpha_{1}, \alpha_{2}, \alpha_{3} \geq 2$ pairwise coprime, $n \geq 2$ and $\langle 1-n, \alpha_{0} \rangle = \mathbb{Z}$; furthermore, assume $\alpha_{3}=1+m\alpha_{1}\alpha_{2}$ for some $m>0$. Then the variety $Y_s$ is a cyclic covering of the smooth affine variety $X$ considered in \cite[Theorem 1.19]{DPO} defined by the equation

\begin{center}
$x + x^n z + y_1^{\alpha_1} + y_2^{\alpha_2} + y_3^{\alpha_3}$
\end{center}

with $f = z \in \mathcal{O}_X (X)$ and $s=\alpha_{0}$. There is a $\mathbb{G}_{m,k}$-action on $X$ given by $(\lambda,x,y_{1},y_{2},y_{3},z) \mapsto (\lambda^{\alpha_{1}\alpha_{2}\alpha_{3}}x,\lambda^{\alpha_{2}\alpha_{3}}y_{1},\lambda^{\alpha_{1}\alpha_{3}}y_{2},\lambda^{\alpha_{1}\alpha_{2}}y_{3},\lambda^{(1-n)\alpha_{1}\alpha_{2}\alpha_{3}}z)$ which makes $f=z$ a quasi-invariant of weight $(1-n)\alpha_{1}\alpha_{2}\alpha_{3}$. As proven in \cite[Theorem 1.19]{DPO}, $X$ is stably $\mathbb{A}^1$-contractible; hence the structure morphism $X \rightarrow Spec(k)$ induces an isomorphism on Chow groups and the coordinate ring $\mathcal{O}_{X}(X)$ is a UFD. The subscheme $F_0 = \{z=0\} \subset X$ is isomorphic to $\mathbb{A}^3_k$ and hence $f=z$ is prime in $\mathcal{O}_{X}(X)$ and $F_0 \rightarrow Spec(k)$ induces an isomorphism on Chow groups. For $i \geq 1$, it follows that $CH^i (F_{0})\otimes_{\mathbb{Z}}R$ and $CH^i (X)\otimes_{\mathbb{Z}}R$ are trivial for any coefficient ring $R$. If $s=\alpha_{1} \in R^{\times}$, Theorem \ref{Chow-Tensor} implies that $CH^i (Y_{s})\otimes_{\mathbb{Z}}R =0$ for $i \geq 1$.\\
It is well-known that algebraic vector bundles over $Y_{s}$ are stably trivial if $CH^1 (Y_s)$, $CH^2 (Y_s)$, $CH^3 (Y_s)$ and $CH^4 (Y_s)$ are trivial because the rank homomorphism induces an isomorphism $K_0 (Y_{s}) \xrightarrow{\cong} \mathbb{Z}$ in this case. It follows from Corollary \ref{rank-hom} that "rationally" all vector bundles over $Y_s$ are indeed stably trivial groups as the rank homomorphism induces an isomorphism $K_{0} (Y_{s}) \otimes_{\mathbb{Z}}\mathbb{Q} \cong \mathbb{Q}$. Corollary \ref{Euler-2} implies that any rank $4$ bundle over $Y_s$ admits a free direct summand of rank $1$.\\
If $k = \mathbb{C}$, then $Y_{s}$ is topologically contractible by \cite[Example 6.2]{Z}: The polynomial $y_{1}^{\alpha_{1}} + x + x^n z^{\alpha_{0}}$ is prime in $\mathbb{C}[x,z,y_{1}]=\mathbb{C}[x,z][y_{1}]$ by Eisenstein's criterion for irreducibility and defines a smooth subvariety of $\mathbb{A}^3_{\mathbb{C}}$; moreover, it is a quasi-invariant of weight $\alpha_{0}\alpha_{1}$ with respect to the $\mathbb{G}_{m,k}$-action $(\lambda,x,y_{1},z) \mapsto (\lambda^{\alpha_{0}\alpha_{1}}x,\lambda^{\alpha_{0}} y_{1},\lambda^{(1-n)\alpha_{1}}z)$. So by \cite[Example 6.2]{Z} the variety $Y_s$ is topologically contractible. Altogether, it follows that $Y_s$ has a trivial Picard group and, in particular, $CH^1 (Y_{s}) \cong CH^1 (Y_{s}) \otimes_{\mathbb{Z}}R = 0$.
%\end{Bsp}
%for the smoothness take Stacks Project Lemma 29.34.14(6) with the affine open cover given by the principal open subsets defined by $1+nx^{n-1}z^{\alpha_{0}}$ and $x^n \alpha_{0} z^{\alpha_{0}-1}$

\subsection{Bicyclic coverings of Koras-Russell threefolds of the first kind}\label{BicyclicFourfold}

%\begin{Bsp}Let $k=\mathbb{C}$.
Consider the variety $Y=Y_{\alpha,\beta}$ defined by the equation

\begin{center}
$x + x^2 (u^{\alpha}+v^{\beta}) + z^{\alpha_2} + t^{\alpha_3}$
\end{center}

with $\alpha, \alpha_{2}, \alpha_{3}, \beta \geq 2$ pairwise coprime. This is a bicyclic covering of the Koras-Russell threefold $X$ of the first kind (see Example \ref{KREX}) defined by the equation

\begin{center}
$x + x^2 y + z^{\alpha_2} + t^{\alpha_3}$
\end{center}

with $f = -y \in \mathcal{O}_X (X)$ of orders $\alpha,\beta$. The discussion of Section \ref{EXKR1} and Theorem \ref{Bicyclic} imply that $CH^{i}(Y_{\alpha,\beta})=0$ for $i \geq 1$. It follows that all vector bundles over $Y$ are stably trivial; in fact, it follows that at least all vector bundles of rank $\neq 2$ are trivial by well-known cancellation theorems (cf. \cite{FRS}, \cite{S2}). All vector bundles of rank $2$ are trivial if and only if the Hermitian $K$-theory group $W_{SL}(Y_{\alpha,\beta})$ is trivial; in fact, the vanishing of $H_{Nis}^{2}(Y_{\alpha,\beta},\textbf{I}^3)$ already guarantees the triviality of all vector bundles over $Y_{\alpha,\beta}$ (cf. \cite{Sy}). Finally, Theorem \ref{Bicyclic-VB} implies that all vector bundles over $Y_{\alpha,\beta}$ are indeed trivial. If $k=\mathbb{C}$, then $Y_{s,t}$ is topologically contractible by Remark \ref{Bicyclic-Contractibility}.

\end{document}